\documentclass[11pt]{article}
\usepackage[top=1in, bottom=1in, left=1.25in, right=1.25in]{geometry}

\usepackage{amsmath,amssymb,amsthm,xcolor,enumerate,bbm,array}
\usepackage[unicode,breaklinks=true,colorlinks=true]{hyperref}
\usepackage{xcolor}
\usepackage{theoremref}
\usepackage{mathrsfs}
\usepackage[]{authblk}
\usepackage{comment}
\usepackage{cancel}
\usepackage{graphicx}
\usepackage{float}
\usepackage{changepage}
\usepackage{array,multirow,makecell}
\setcellgapes{1pt}
\makegapedcells
\newcolumntype{R}[1]{>{\raggedleft\arraybackslash }p{#1}}
\newcolumntype{L}[1]{>{\raggedright\arraybackslash }p{#1}}
\newcolumntype{C}[1]{>{\centering\arraybackslash }p{#1}}

\newcommand{\crm}[1]{{\color{magenta} #1}} 
\renewcommand{\crm}[1]{#1}

\numberwithin{equation}{section}
\newtheorem{thm}{Theorem}[section]

\newtheorem{lem}[thm]{Lemma}
\newtheorem{prop}[thm]{Proposition}
\newtheorem{defn}[thm]{Definition}

\theoremstyle{remark}
\newtheorem{remark}{Remark}[section]
\theoremstyle{definition}

\newcommand{\R}{\mathbb{R}}

\renewcommand{\div}{\mathop{\rm div}\nolimits}
\newcommand{\curl} {\mathop{\rm curl}\nolimits}

\newcommand{\EQS}[1]{\begin{equation}\begin{split} #1 \end{split}\end{equation}}
\newcommand{\EQs}[1]{\begin{equation*}\begin{split} #1 \end{split}\end{equation*}}

\allowdisplaybreaks
\begin{document}

\title
{Global existence and uniqueness of  weak solutions  for the MHD equations  with large $L^{3}$-initial values}
\author[1]{\rm Baishun Lai}
\author[2]{\rm Ge Tang}
\author[3]{\rm Ziying Xu}
\affil[1]{\footnotesize MOE-LCSM, School of Mathematics and Statistics, Hunan Normal University, Changsha, Hunan 410081, P. R. China}
\affil[2,3]{\footnotesize School of Mathematics and Statistics, Hunan Normal University, Changsha, Hunan 410081, P. R. China}
\date{}
\maketitle

\begin{abstract}
This paper is concerned with the weak solution theory for the MHD system with large $L^3$-initial data.
Due to the fact that the natural boundary condition on the magnetic field $H$ is the \emph{slip} boundary condition, the Leray-Schauder fixed-point theorem, which have used to investigate the weak solution theory of the Navier-Stokes system, becomes invalid. To address such difficulty, we will invoke  the Leray's approximation technique and the perturbation theory to seek a global weak solution to the Cauchy problem for MHD equations with large $L^3$-initial data. Our strategy provides  a simple alternative (self-contained)  proof of  weak $L^3$-solution theory of incompressible Navier-Stokes system. Moreover, this weak solution is unique under some restrictions.

\medskip

\end{abstract}

\setcounter{equation}{0}
 \setcounter{equation}{0}
\section{Introduction}
This paper is concerned with the existence and uniqueness of weak solutions for the three-dimensional incompressible Magnetohydrodynamics (MHD for short) equations in $\mathbb R^3\times(0,+\infty)$ 
		\begin{equation}\label{EQ-tlh}
		\left\{
		\begin{array}{l}
			\begin{aligned}
				&\partial_{t}v-\Delta v+v\cdot\nabla v+\nabla\Pi=H\cdot\nabla H,\\
				&\partial_{t}H-\Delta H+v\cdot\nabla H=H\cdot\nabla v,\\
				&\div v= \div H=0,
			\end{aligned}
		\end{array} \right.
	\end{equation}
 complemented with the initial condition	
\begin{equation}\label{initial value}
		v(x,0)=v_0,\quad H(x,0)=h_0 \quad\text{in} \quad\mathbb R^3.
	\end{equation}
Here, $v$ is a velocity field, $H$ is a
magnetic field, $\Pi$ stands for the total pressure of the fluid. Magnetohydrodynamics  is a branch of continuum mechanics that studies the flow of electrically conducting fluids under the influence of magnetic fields. It holds a core position in plasma physics, astrophysics, and geophysics. Conducting fluids not only exhibit fluid characteristics but also carry charges under the influence of magnetic fields, making electromagnetic effects and fluid dynamics closely intertwined in their motion patterns. The MHD equations serve as the fundamental tool for describing this interaction by combining fluid dynamics equations with Maxwell's equations to trace the changes in fluid motion and the magnetic fields they produce in ionized media.

\crm{Notice that}, in the case $H\equiv0$, \eqref{EQ-tlh} reduces to the incompressible Navier-Stokes equations in $\R^3\times (0,\infty)$
\begin{equation}\label{ns}
	\left\{
	\begin{array}{l}
		\begin{aligned}
			&\partial_{t}v-\Delta v+v\cdot\nabla v+\nabla p=0, \\
			&\div v=0,\\
            & v(x,0)=v_0, \quad \quad  x\in\R^3
		\end{aligned}
	\end{array} \right.
\end{equation}
which has a simple form but rich mathematical structure. In the groundbreaking work \cite{leray1934}, Leray  constructed some global weak solutions to the system \eqref{ns} for $v_0\in L^2(\R^3)$, called nowadays Leray-Hopf weak solutions. However, whether or not  Leray-Hopf weak solutions remain smooth for all times is one of the most famous open problems in mathematics. In the literature, there is a large body of classical results on the regularity theory of weak solutions, we refer interested readers to see \cite{Galdi,Lemarie-Rieusset,leray1934,oseen1911formules,3-D ns equations,RT} and reference therein. Turning to MHD system \eqref{EQ-tlh}-\eqref{initial value}, its  existence and uniqueness theory is closely related to that of the Navier-Stokes equations \eqref{ns}. The earliest work on this direction (existence and uniqueness theory of solution) can be traced back to the classic work of Duvaut and Lions \cite{Duvaut1972} in 1972, where they obtained global weak solutions and local strong solutions for equations \eqref{EQ-tlh}-\eqref{initial value} under certain initial and boundary conditions by using Faedo-Galerkin approximation techniques. Subsequently, Sermange and Temam \cite{Temam} extended these results to the Cauchy problem via the regularity theory of the Stokes operator and energy methods. By now, there have been many diverse and interesting results of MHD system \eqref{EQ-tlh}-\eqref{initial value} on this direction (see, e.g \cite{cao2011,Chemin2016,Fer2021,suen2020}), but we are not  pursuing such directions in this paper.

Recently, more and more mathematicians have become aware of the importance of the theory of the weak solution  in the study of fluid equations, where the initial value belongs to a more general class that includes $L^3(\mathbb{R}^3), L^2_{uloc}(\mathbb{R}^3)$, etc, rather than in energy class like $L^2(\mathbb{R}^3)$. For example,  the theory of local energy weak solutions of \eqref{ns} for $v_0\in L^2_{uloc}(\mathbb{R}^3)$, which was first introduced by Lemari\'{e}-Rieusset \cite{Lemarie}, plays a key role in the construction large forward self-similar solutions of \eqref{ns}, see \cite{jia2014} for more details. To the best of our  knowledge, the earliest work in this direction can be traced back to Calder\'{o}n's work \cite{Calderon}, where the author considered  the global weak solution theory of system \eqref{ns} with initial data $v_0\in L^p \,(2<p<3)$. Due to Calder\'{o}n's work, the ``gap" between the Hopf-Leray theory ($p=2$) and that of Kato ($p\geq 3$) is bridged.  Very recently,  Seregin and  \v{S}ver\'{a}k in \cite{seregin2016} considered the global weak solution theory for  large initial data in $L^3(\mathbb R^3)$. Precisely, they used the Leray-Schauder degree theory to construct a large weak solution, called as the large global $L^3$- weak solution (see below for the precise definition), and demonstrated that this weak solution is uniqueness under some restrictions. Besides,  another key ingredient in \cite{seregin2016} is the decomposition technique to the initial data (see below for the precise statement), which has been used in Calder\'{o}n's work. \crm{Besides scaling invariance of the Navier-Stokes system \eqref{ns} on $L^3$, the main reason one develops the $L^3$- weak solution theory is as follows: The global weak $L^3$ is stable in the following sense:  given a sequence of global weak $L^3$ solutions with initial data $u_0^{(n)}\rightharpoonup u_0$ in $L^3$, there exists a subsequence converging in the sense of distributions to a global
weak $L^3$ solution with initial data $u_0$. This property plays a distinguished role  in the following two aspects : (i) The regularity theory of the Navier-Stokes equations \eqref{ns} or MHD system \eqref{EQ-tlh}. For example,  such sequences of solutions arise naturally when zooming in on a potential singularity of the  Navier-Stokes equations \eqref{ns}, as in the paper \cite{L. Escauriaza} by Escauriaza, Seregin, and \v{S}ver\'{a}k.
(ii) This property enables us to construct a  global  weak $L^{3,\infty}$ solution with a large $L^{3,\infty}$ data, for details see \cite{Barker2018}. In many cases it is desirable to have a good theory of the large global  weak $L^{3,\infty}$ solution for initial data $u_0\in L^{3,\infty}$. First, this solution contains scale-invariant solutions (also called as forward self-similar solutions) investigated by Jia and \v{S}ver\'{a}k in \cite{jia2014}. Secondly, a local-in-time mild solution, so far, is not known  to exist under the framework $L^{3,\infty}$, unless the initial value $\|u_0\|_{L^{3,\infty}}$ is sufficiently small. It appears that the Cauchy problem (1.3) is ill-posed with large $L^{3,\infty}$ initial date, see for example \cite{jia2015}. Thirdly, it also provides  a simple alternative approach to Lemari\'{e}-Rieusset's theory \cite{Lemarie} of local energy solutions in the case $u_0\in L^{3,\infty}$.
}

To illuminate the motivations of this paper in detail, we now sketch Seregin-\v{S}ver\'{a}k's strategy as follows. Since the initial datum $v_0\in L^{3}(\R^3)$ is not in the  energy class, one can not directly construct a large weak solution  of Navier-Stokes system \eqref{ns}. To derive the desired result,  Seregin and \v{S}ver\'{a}k in \cite{seregin2016} adopted the ``initial value homogeneous" argument to fix this difficulty. Precisely, let $v = e^{t\Delta}v_0 + u$,  then  seeking a weak solution is equivalent to find a weak solution of
\EQS{\label{ns-01}
\left.\begin{array}{ll}
\partial_tu-\Delta u+\nabla p+v\cdot\nabla v=0\\
\div u=0\\
\end{array}\right\}\,\, \mbox{in}\quad \R^3\times(0,+\infty),
}
supplemented with an initial condition
\EQS{\label{ns-02}
u(x,0)=0, \quad\quad \mbox{in}\quad \R^3.
}
The main tool used in \cite{seregin2016} is \crm{the {\em Leray-Schauder fixed-point theorem}} proposed by Schauder in 1927 and developed by Leray in 1933. Precisely,  exploiting the Leray-Schauder principle to find a solution, one need to reduce the study of the problem \eqref{ns-01}-\eqref{ns-02} into that of the corresponding integral system:
\EQS{
u=\mathcal{A}(e^{t\Delta}v_0 + u)
}
To solve this integral system, it suffices to verify  $\mathcal{A}$ satisfies all the requirements of the Leray-Schauder principle in some selected Banach space $X$ (Here the Banach space $X$ is chosen as $L^2(0,T; L_{\sigma}^2(\R^3))$:

(i)\; $\mathcal{A}: X\buildrel{}\over{-\!\!\!\!\rightarrow}X$ is a continuous and compact operator.

(ii)\; There exists a constant $C$ such that, for every $\lambda\in[0,1],$ $$u=\mathcal{A}(e^{t\Delta}v_0 + u)\;\Longrightarrow\; \|u\|_{X}\leq C.$$

In order to derive the  compactness of  $\mathcal{A}$, the authors in  \cite{seregin2016} had to  consider the system \eqref{ns-01}-\eqref{ns-02} in $B(R)\times(0,T)$\footnote{Throughout this paper, we denote $B(x,r):=\{y\in\R^3: |y-x|<r$\}, and  write $B(x,r)$ as $B(r)$ for simplicity.} under zero boundary condition, where the compactness of  $\mathcal{A}$ can be obtained by the  fact that the embedding $H^{1}(B(R))\hookrightarrow L^2(B(R))$ is compact. Besides, the classical energy technique ensures that the requirement (ii), i.e., the so called a priori estimate,  is valid. Once the requirements (i) and (ii) are verified, one can construct a solution, say $u$, of the system \eqref{ns-01}-\eqref{ns-02} in $B(R)\times(0,T)$ in the Banach space $X$ by using the Leray-Schauder fixed-point theorem. Moreover, the solution $u$ fulfills
$$
\|u\|_{L^{\infty}(0,T; L^2(B(R)))}+\|\nabla u\|_{L^{2}(0,T; L^2(B(R)))}\leq C(\|v_0\|_{L^{3}(\R^3)})\sqrt{T}.
$$
Let $R\to\infty$, one eventually derives a weak solution of the system \eqref{ns-01}-\eqref{ns-02} in $\R^3\times(0,T)$ for any $T>0$.

Now, turning to the  MHD equations \eqref{EQ-tlh}-\eqref{initial value},  it is remarked that this system in $B(R)\times(0,T)$ coupled with the \emph{no-slip} boundary condition
$$
v(x,t)=H(x,t)=0 \quad\quad \mbox{on}\quad \partial B(R)\times(0,T)
$$
is overdetermined. Usually, in the context of the MHD equations in bounded domain, the natural boundary condition on the magnetic field $H$ is the \emph{slip} boundary condition
 \EQs{
H(x',0,t)\cdot\nu=0,\quad
\curl H(x,t)\times \nu=0\quad \text{ on }\ \partial B(R)\times(0,T),
}
where $\nu$ is the outer normal on $\partial B(R)$.  Therefore, it seems that one, as in Navier-Stokes system, cannot use Leray-Schauder fixed-point theorem  to construct a solution to the MHD equations. In the present paper, To address this difficulty, we will invoke two classical techniques: the Leray approximation technique and the perturbation theory,  to seek a large global weak $L^3$-solution (see below for the precise definition) of \eqref{EQ-tlh}-\eqref{initial value} in $\R^3\times (0,T)$. The key ingredient of our argument is that it allows for a direct investigation of the \eqref{EQ-tlh}-\eqref{initial value} in $\R^3\times (0,T)$, avoiding the need to consider in $B(R)\times(0,T)$. \crm{Besides, it is worth noting  that our strategy provides  a simple alternative (self-contained)  proof of  weak $L^3$-solution theory of incompressible Navier-Stokes system.}
 Let us now state our first theorem.

\begin{thm}\label{thm 1.1}
Let $(v_0,h_0)\in L_{\sigma}^3(\mathbb R^3)$, the Cauchy problem \eqref{EQ-tlh}-\eqref{initial value} admits at least one global weak $L^3$-solution in the sense of Definition \ref{def1}.
\end{thm}

In fact, the global weak $L^3$-solution constructed in the Theorem \ref{thm 1.1} is unique under some restrictions

\begin{thm}\label{thm 1.2}
Let $(v,H)$ and $(\tilde v,\tilde H)$ be two global weak $L^3$-solutions to the Cauchy problem \eqref{EQ-tlh}-\eqref{initial value} corresponding to the initial data $(v_0,h_0)\in L_{\sigma}^3(\mathbb R^3)$, and $(v,H)\in L^\infty(0,T;L^3(\mathbb R^3))$. \crm{If  there exist two absolute constants $\mu_1,~\mu_2>0$} such that for some number $0<T_1<T$,
\begin{equation} \label{miu}
\|v-v_0\|_{L^\infty(0,T_1;L^3(\mathbb R^3))}\leq\crm{\mu_1},\quad
\|H-h_0\|_{L^\infty(0,T_1;L^3(\mathbb R^3))}\leq\crm{\mu_2},
\end{equation}
then $(v,H)=(\tilde v,\tilde H)$ in $\mathbb R^3\times(0,T).$
\end{thm}

Finally, we conclude this section by presenting the organization of this paper:  Section 2 contains preliminaries which consist of some necessary notations and some useful lemmas.  Section 3 is devoted to the proof of the local in time as well as global in time existence results. Section 4 is devoted to the proof of Theorem \ref{thm 1.2}.
\medskip

\setcounter{equation}{0}
 \setcounter{equation}{0}
\section{Preliminary}
\subsection{Functional spaces and several useful lemmas}
Let $a, b$ be \crm{two extended real numbers}, $-\infty\leq a<b\leq\infty$,
and let $X$ be a Banach space. For given $p, 1\leq p\leq\infty$, $L^p(a, b; X)$ denotes the space of $L^p$-Lebesgue
integrable functions from $[a, b]$ into $X$, which is a Banach space with the norm
	\begin{equation*}
		\|v\|_{L^p(a, b ; X)}=\left\{\begin{array}{ll}
		\left(\int_a^b\|v\|_X^p \mathrm{~d} t\right)^{\frac{1}{p}} & \text {if}\quad 1 \leq p<\infty,\\
		\text { ess } \sup\limits_{t \in[a, b]}\|v\|_X & \text {if} \quad p=\infty.
		\end{array}\right.
	\end{equation*}
\crm{With this notation, we define $L^p(\R^3\times(a,b))\triangleq L^p(a,b;L^p(\R^3))$ with $p\in[1,\infty]$.} The Sobolev space $W^{k,p}(\Omega)$ is a space of functions on a domain $\Omega\subseteq\mathbb R^3$
such that the function and its weak derivatives up to order $k$ are in the $L^p(\Omega)$ space. The space is equipped with the norm:
	\begin{equation*}
		\|v\|_{W^{k, p}(\Omega)}=\left\{\begin{array}{ll}
		\left(\sum\limits_{\alpha \leq k} \int_{\Omega}\left|D^\alpha v\right|^p \mathrm{~d} x\right)^{\frac{1}{p}} &\text {if}\quad 1 \leq p<\infty, \\
		\sum\limits_{\alpha \leq k} \operatorname{ess} \sup\limits  _{\Omega}\left|D^\alpha v\right| &\text {if}\quad p=\infty .
		\end{array}\right.
	\end{equation*}
The space $C_{0}^{\infty}(\Omega)$ is constituted by all infinitely differentiable functions with compact support in $\Omega$. In particular, we usually write
$$
L^{p}_{\sigma}(\Omega)=\left\{v\in L^{p}(\Omega): \int_{\Omega}v\cdot\nabla\varphi \,{\rm d}x=0, \ \mbox{for any}\ \varphi\in C_{0}^{\infty}(\Omega)\right\}.
$$
Next, we provide the definition of \crm{global} weak $L^3$-solutions of  the Cauchy problem \eqref{EQ-tlh}-\eqref{initial value}.
\begin{defn}\label{def1}
Let $(v_0,h_0)\in L^3_{\sigma}(\mathbb R^3)$. The pair $(v,H)$ is called a \crm{global} weak $L^3$-solutions to problem \eqref{EQ-tlh}-\eqref{initial value} if
		
{\rm (a)} $v=e^{t\Delta}v_0+v_2\triangleq v_1+v_2$, $H=e^{t\Delta}h_0+H_2\triangleq H_1+H_2$  with  $(v_2,H_2)\in L^\infty(0,\infty;L^2(\mathbb R^3))\cap L^2(0,\infty;\dot{H}^1(\mathbb R^3))$;
	
{\rm (b)} The pair $(v_2,H_2)$ satisfies the following systems in $\mathbb R^3\times(0,+\infty)$
\begin{equation}\label{v_2}
\left\{
\begin{array}{l}
\begin{aligned}
&\partial_{t}v_2-\Delta v_2+\nabla\Pi=-v\cdot\nabla v+H\cdot\nabla H,\\
&\partial_{t}H_2-\Delta H_2=-v\cdot\nabla H+H\cdot\nabla v,\\
&\div v_2=\div H_2=0,\\
&v_2(\cdot,0)=0,\, H_2(\cdot,0)=0,
\end{aligned}
\end{array}\right.
\end{equation}
in the sense of distributions. Here, $\Pi\in L^\frac{3}{2}_{\rm loc}(0,\infty;L^\frac{3}{2}(\mathbb R^3))$;
		
{\rm (c)} The pair $(v_2,H_2)$ satisfies the \crm{global energy} inequality
		\begin{equation}\label{global energy inequality}
			\begin{aligned}
				&\frac{1}{2} \int_{\mathbb R^3}(\left|v_2(x,t)\right|^2+\left|H_2(x, t)\right|^2)\,{\rm d}x+\int_0^t \int_{\mathbb R^3}(\left|\nabla v_2\right|^2+\left|\nabla H_2\right|^2)\,{\rm d}x{\rm d}s\\
				&\leq\int_0^t \int_{\mathbb R^3}(v_1\otimes v:\nabla v_2-H_1\otimes H:\nabla v_2-v_1\otimes H:\nabla H_2+H_1\otimes v: \nabla H_2)\,{\rm d}x{\rm d}s,
			\end{aligned}
		\end{equation}
		for all $t\in(0,T]$, \crm{with any $T<\infty$};
		
{\rm (d)} The following local energy inequality holds for the pair $(v_2,H_2)$
		\begin{equation}\label{local energy estimate}
			\begin{aligned}
				&\int_{\mathbb{R}^3} \phi(\cdot, t)(\left|v_2(\cdot, t)\right|^2+\left|H_2(\cdot, t)\right|^2)\,{\rm d}x+2 \int_0^t \int_{\mathbb{R}^3} \phi(\left|\nabla v_2\right|^2+\left|\nabla H_2\right|^2)\,{\rm d}x{\rm d}s \\
				&\leq\int_{0}^{t} \int_{\mathbb R^3}(\left|v_2\right|^2+\left|H_2\right|^2)(\Delta \phi+\partial_t \phi)\,{\rm d}x{\rm d}s+\int_{0}^{t}\int_{\mathbb R^3}(|v_2|^2+|H_2|^2+2\Pi)v_2\nabla\phi\, {\rm d}x{\rm d}s\\
				&\quad-2\int_{0}^{t}\int_{\mathbb R^3}(v_2H_2)(H_2\cdot\nabla\phi)\,{\rm d}x{\rm d}s\\
				&\quad+\int_{0}^{t}\int_{\mathbb R^3}\big(v_1\otimes v_2+v_2\otimes v_1+v_1\otimes v_1-H_1\otimes H_2-H_2\otimes H_1-H_1\otimes H_1\big)\\
				&\quad\quad(\nabla v_2\phi+v_2\otimes\nabla\phi)\,{\rm d}x{\rm d}s\\
				&\quad+\int_{0}^{t}\int_{\mathbb R^3}\big(H_1\otimes v_2+H_2\otimes v_1+H_1\otimes v_1-v_1\otimes H_2-v_2\otimes H_1-v_1\otimes H_1\big)\\
				&\quad\quad(\nabla H_2\phi+H_2\otimes\nabla\phi)\,{\rm d}x{\rm d}s,
			\end{aligned}
		\end{equation}
		for any non-negative function $\phi\in C_0^{\infty}(\mathbb R^3\times(0,T))$.
	\end{defn}

\medskip

\begin{remark}\label{remark}
\crm{Usually, the pair $(v_1,H_1)=e^{t\Delta}(v_0, h_0)$ is called as the caloric extension, which fulfills the following nice properties}
 (for example see \cite{Galdi})
\EQS{\label{v_1}
		\|(v_1,H_1)\|_{L^\infty(0,\infty;L^3(\mathbb R^3))}+\|(v_1,H_1)\|_{L^5(0,\infty;L^5(\mathbb R^3))}
&+\|(v_1,H_1)\|_{L^8(0,\infty;L^4(\mathbb R^3))}\\&\leq C\|(v_0,h_0)\|_{L^3(\mathbb R^3)},
}
and
\EQs{
\lim_{t\to0}\|v_1(x,t)-v_0\|_{L^3(\R^3)}=0,\quad\quad \lim_{t\to0}\|H_1(x,t)-h_0\|_{L^3(\R^3)}=0.
}
\end{remark}

\crm{\begin{remark}
For the global energy inequality \eqref{global energy inequality}, one can not put $T=\infty$, since the quantity
$$
\int_{\mathbb R^3}(\left|v_2(x,t)\right|^2+\left|H_2(x, t)\right|^2)\,{\rm d}x+\int_0^t \int_{\mathbb R^3}(\left|\nabla v_2\right|^2+\left|\nabla H_2\right|^2)\,{\rm d}x{\rm d}s
$$
may be blow up as $t\to\infty$,  for details please see Lemma \ref{Energy Estimate}.
\end{remark}}

\begin{remark}
The local energy inequality \eqref{local energy estimate} demonstrates that the \crm{global} weak $L^3$-solutions are suitable in the sense of Caffarelli, Kohn and Nirenberg \cite{CKN}, \crm{which enable us to apply the very rich $\varepsilon$-regularity theory developed by Caffarelli-Kohn-Nirenberg \cite{CKN} to them}.
\end{remark}

Let us now recall a classical fixed point theorem which will be used to construct  the approximate solution sequence of \eqref{v_1}. For completeness, we give a proof here.

\begin{lem}[Fixed point theorem]\label{fixed point thm}
Let $X$ be a Banach space, and define the operator $\mathcal{M}$ as
$$
\mathcal{M}(u,v)=\mathcal{B}(u,v)+\mathcal{L}(u)+\mathcal{R}(x,t), \quad u,v\in X,
$$
where $\mathcal{B}:X\times X\to X $ is a bounded bilinear operator such that
$$
\|\mathcal{B}(u,v)\|_{X}\leq c_1\|u\|_{X}\|v\|_{X},
$$
 $\mathcal{L}: X\to X $ is a bounded linear operator such that
 $$
 \|\mathcal{L}(u)\|_{X}\leq c_2\|u\|_{X},
 $$
  and $\mathcal{R}(x,t)$ is a measurable function on $X$.  Moreover,
 $c_1,\ c_2$ and $\|\mathcal{R}\|_{X}$ satisfy
\begin{equation}\label{eq1.4}
(1-c_2)^2>4c_1\|\mathcal{R}\|_{X},\quad c_1>0,\quad 0\leq c_2<1.
\end{equation}
Then, the quadratic operator $\mathcal{M}(u,u)$ has a fixed point $\bar{u}$  such that
\begin{equation*}
\|\bar{u}\|_{X}\leq x_1=\frac{1-c_2-\sqrt{(c_2-1)^2-4c_1\|\mathcal{R}\|_{X}}}{2c_1}.
\end{equation*}
\end{lem}

\begin{proof}
First,  notice that the term $\sqrt{(c_2-1)^2-4c_1\|\mathcal{R}_{X}\|}$ makes sense due to \eqref{eq1.4}. Thus, the line $Y_1(x)=x$ and the parabola $Y_2(x)=c_1x^2+c_2x+\|\mathcal{R}\|_{X}$ intersect as $x_1$ and $x_2\, (0<x_1<x_2)$ with
$$
x_1=\frac{1-c_2-\sqrt{(c_2-1)^2-4c_1\|\mathcal{R}\|_{X}}}{2c_1},\quad x_2=\frac{1-c_2+\sqrt{(c_2-1)^2-4c_1\|\mathcal{R}\|_{X}}}{2c_1}.
$$
On the other hand, $Y'_2(x)=2c_1x+c_2>0$  when $0<x\leq x_1$, thus  $Y_2(x)\leq Y_2(x_1)=x_1$.  This implies
\begin{equation*}
\begin{split}
\|\mathcal{M}(u,u)\|_{X}&\leq c_1\|u\|_{X}^2+c_2\|u\|_{X}+\|\mathcal{R}\|_{X}\\
&\leq c_1x_1^2+c_2x_1+\|\mathcal{R}\|_{X}\\
&\leq x_1,
\end{split}
\end{equation*}
for any $u\in B=\{u: \|u\|_{X}\leq x_1\}$. Therefore, the operator $\mathcal{M}(u,u)$ maps the ball $B$ to itself. Next we claim that the operator $\mathcal{M}(u,u)$ is a contractive mapping. As a matter of fact, we see that
\begin{equation*}
\begin{split}
\|\mathcal{M}(u_1,u_1)-\mathcal{M}(u_2,u_2)\|_{X}&=\|\mathcal{B}(u_1,u_1)-\mathcal{B}(u_2,u_2)+\mathcal{L}(u_1-u_2)\|_{X}\\
&\leq \|\mathcal{B}(u_1,u_1-u_2)\|_{X}+\|\mathcal{B}(u_1-u_2,u_2)\|_{X}+\|\mathcal{L}(u_1-u_2)\|_{X}\\
&\leq c_1\|u_1\|_{X}\|u_1-u_2\|_{X}+c_1\|u_2\|_{X}\|u_1-u_2\|_{X}+c_2\|u_1-u_2\|_{X}\\
&\leq \big[c_1(\|u_1\|_{X}+\|u_2\|_{X})+c_2\big]\|u_1-u_2\|_{X}\\
&\leq (2c_1x_1+c_2)\|u_1-u_2\|_{X}\\
&\leq\gamma\|u_1-u_2\|_{X},
\end{split}
\end{equation*}
with $\gamma=2c_1x_1+c_2<1$ due to the fact
\begin{align*}
\gamma&=2c_1x_1+c_2\\
&=2c_1\frac{1-c_2-\sqrt{(c_2-1)^2-4c_1\|\mathcal{R}\|_{X}}}{2c_1}+c_2\\
&=1-\sqrt{(c_2-1)^2-4c_1\|\mathcal{R}\|_{X}}.
\end{align*}
Thus, $\mathcal{M}(u,u)$ is a contractive mapping on the ball $B$,  and consequently it has a fixed point $\bar{u}$ on the ball $B$ such that
\begin{equation*}
\|\bar{u}\|_{X}\leq x_1=\frac{1-c_2-\sqrt{(c_2-1)^2-4c_1\|\mathcal{R}\|_{X}}}{2c_1},
\end{equation*}
which finishes the proof.
\end{proof}

We now present the following \crm{two lemmas}, which will be used in the proof of Theorem \ref{TH3.3}.
\begin{lem}[\cite{3-D ns equations}]\label{3-D ns equations}
If $u\in L^2(0,T;\dot{H}^1(\mathbb R^3)),\partial_t u\in L^2(0,T;{H}^{-1}(\mathbb R^3))$ with $T\in(0,+\infty]$, then we have $u\in C([0,T];L^2(\mathbb R^3))$ with
$$\sup_{t\in [0,T]}\|u(t)\|_{L^2(\mathbb R^3)}\leq C\big(\|u\|_{L^2(0,T;\dot{H}^1(\mathbb R^3))}+\|\partial_t u\|_{L^2(0,T;{H}^{-1}(\mathbb R^3))}\big).$$
\end{lem}

\medskip

\begin{lem}\label{H^-1 to X_T}
	If $g\in L^2(0,\infty;\dot{H}^{-1}(\mathbb R^3))$, then $\crm{U(x,t)\triangleq\int_{0}^{t}\int_{\R^3}W_{(t-s)}(x-y)\mathbb Pg(y)\,{\rm d}y{\rm d}s}\in C(0,\infty;L^2(\mathbb R^3))\cap L^2(0,\infty;\dot{H}^1(\mathbb R^3))$ and
$$
\|U\|_{ L^\infty(0,\infty;L^2(\mathbb R^3))\cap L^2(0,\infty;\dot{H}^1(\mathbb R^3))}\leq C\|g\|_{L^2(0,\infty;\dot{H}^{-1}(\mathbb R^3))}.
$$
Here, $W_t(x)=\frac{1}{(4\pi t)^{\frac{3}{2}}}e^{-\frac{|x|^2}{4t}}$ and $\mathbb{P}$ is Leray-Hopf projection operator in $\R^3$ defined as
\begin{equation*}
(\mathbb{P} f)_i = f_i + \partial_i (-\Delta)^{-1} \nabla \cdot f, \quad \forall ~f\in C_{0}^{\infty}(\R^3).
\end{equation*}
\end{lem}
\begin{proof}
First, by the Plancherel equality, one has for any $t\in(0,+\infty)$
\begin{equation*}
\begin{aligned}
\|U\|^2_{L^2(\mathbb R^3)}&= \left\|\int_{0}^{t}e^{-(t-s)\xi^2}
\Big(\hat{g}_i-\frac{\xi_i\xi_j}{|\xi|^2}\hat{g}_j\Big)\,{\rm d}s\right\|^2_{L^2(\mathbb R^3)}\\
&\leq\int_{\mathbb R^3}\int_{0}^{t}e^{-2(t-s)\xi^2}\,{\rm d}s\int_{0}^{t}|\hat{g}|^2\,{\rm d}s{\rm d}\xi\\
&=\int_{\mathbb R^3}\frac{1}{2\xi^2}(1-e^{-2t\xi^2})\int_{0}^{t}|\hat{g}|^2\,{\rm d}s{\rm d}\xi\\
&\leq\frac{1}{2}\int_{0}^{\infty}\int_{\mathbb R^3}\xi^{-2}|\hat{g}|^2\,{\rm d}\xi{\rm d}s\\
&\leq\frac{1}{2}\|g\|^2_{L^2(0,\infty;\dot{H}^{-1}(\mathbb R^3))}.
\end{aligned}
\end{equation*}
Here, $\hat{g}$ denotes by the Fourier transform of $g$ in $\R^n$.  Therefore,
$$
\|U\|_{L^\infty(0,\infty;L^2(\mathbb R^3)) }\leq\|g\|_{L^2(0,\infty;\dot{H}^{-1}(\mathbb R^3))}.
$$
On the other hand, by using the Plancherel equality again, one has
\begin{equation*}
\begin{aligned}
\|\nabla U\|_{L^2(0,\infty;L^2(\mathbb R^3))}&\leq\Big(\int_{0}^{\infty}\int_{\mathbb R^3}\Big|\frac{i\xi}{i\tau+\xi^2}|^2|\hat{g}\Big|^2\,{\rm d}\xi{\rm d}\tau\Big)^{\frac{1}{2}}\\
&=\Big(\int_{0}^{\infty}\int_{\mathbb R^3}\frac{\xi^4}{\tau^2
+\xi^4}\xi^{-2}|\hat{g}|^2\,{\rm d}\xi {\rm d}\tau\Big)^{\frac{1}{2}}\\
&\leq\Big(\int_{0}^{\infty}\int_{\mathbb R^3}\xi^{-2}|\hat{g}|^2\,{\rm d}\xi {\rm d}\tau\Big)^{\frac{1}{2}}\\
&\leq \|g\|_{L^2(0,\infty;\dot{H}^{-1}(\mathbb R^3))}.
\end{aligned}
\end{equation*}
From this, it is clear that
$$
\|U\|_{L^2(0,\infty;\dot{H}^{1}(\mathbb R^3))}\leq\|g\|_{L^2(0,\infty;\dot{H}^{-1}(\mathbb R^3))}.
$$
\crm{Next,} we claim that $\Delta U\in L^2(0,\infty;\dot{H}^{-1}(\mathbb R^3))$.
\crm{In fact,}
\begin{align*}
&\|\Delta U\|^2_{L^2(0,T;H^{-1}(\mathbb R^3))}=\int_{0}^{T}\|\Delta U\|^2_{H^{-1}(\mathbb R^3)}\,{\rm d}t\\
&=\int_{0}^{T}\Big(\sup_{\varphi\in H^{1}(\mathbb R^3)}\frac{|\langle\Delta U,\varphi\rangle|}{\|\varphi\|_{H^{1}(\mathbb R^3)}}\Big)^2\,{\rm d}t\\
&\leq \int_{0}^{T}\|\nabla U\|^2_{L^2(\mathbb R^3)}\,{\rm d}t\leq \|U\|^2_{L^2(0,T;\dot {H}^1(\mathbb R^3))}.
\end{align*}
It is clear that $U$ is a solution to \crm{the following heat equation in $\R^3\times(0,+\infty)$}
\begin{equation*}
\left\{
\begin{array}{l}
\begin{aligned}
&\partial_t u-\Delta u=\mathbb Pg,\\
&u(x,0)=0,
\end{aligned}
\end{array} \right.
\end{equation*}
and we already know that $\Delta U$ and $\mathbb Pg$ belong to $L^2(0,\infty;H^{-1}(\mathbb R^3)$, from which we can deduce that $\partial_t U\in L^2(0,\infty;H^{-1}(\mathbb R^3))$. \crm{This, along with the Lemma \ref{3-D ns equations}, derives} $U\in C(0,\infty;L^2(\mathbb R^3)).$
\end{proof}

We conclude this subsection by introducing a compactness lemma known as the Aubin-Lions lemma, which enables us to obtain a \crm{global} weak $L^3$-solution from the approximate solution sequence.

\begin{lem}[\cite{BV}]\label{RTtheorem_23}
Let $X\subset Y \subset Z$ be separable, reflexive Banach spaces, such that the embedding $X\subset Y$ is compact and the embedding $Y \subset Z$ is continuous. Let $T>0$, and assume that we have a sequence of functions $\{u_n\}_{n\geq1}$ such that
$$
\{u_n\}\,\mbox{is uniformly bounded in } L^p(0,T;X),
$$
$$
\{\partial_tu_n\}\,\mbox{is uniformly bounded in } L^q(0,T;Z)
$$
for $p,q>1$. Then the sequence $\{u_n\}$ is  precompact in $L^p(0,T;Y)$.
\end{lem}

\subsection{Several useful inequalities}
In this subsection, we introduce some important inequalities which play an important role in the proof of this paper. To begin with, we introduce the well-known Gronwall's inequality:

\begin{lem}[Gronwall's inequality \cite{3-D ns equations}]\label{Gronwall inequality}
{\rm(a)} Let $\eta:[0, T] \rightarrow[0, \infty)$ be an absolutely continuous function that satisfies the differential inequality
$$
\eta^{\prime}(t) \leq \phi(t) \eta(t)+\psi(t)
$$
where $\phi$ and $\psi$ are non-negative integrable functions. Then
$$
\eta(t) \leq \mathrm{e}^{\crm{\int_0^{t} \phi(s) \mathrm{d} s}}\left[\eta(0)+\int_0^t \psi(s) \mathrm{d} s\right] \quad \text { for all } t \in[0, T].
$$
{\rm(b)} Let $\eta:[0, T] \rightarrow[0, \infty)$ be a continuous function satisfying
$$
\eta(t) \leq a(t)+\int_0^t\phi(s)\eta(s)\,{\rm d}s,
$$
where $a,\phi:[0, T] \rightarrow[0, \infty)$  are integrable functions and $a$ is increasing. Then
$$
\eta(t) \leq a(t)\exp\Big(\int_0^t \phi(s)\,{\rm d} s\Big) \quad \text { for all } t \in[0, T].
$$
\end{lem}
\begin{lem}[Sobolev-Poincere type inequality]\label{sobolev-poincare inequality}
	Let \( \Omega \subset \mathbb{R}^n \) be a bounded domain and \(1\leq p<n \), there exists \crm{a positive constant $C$ depending only on $\Omega$} such that for all measurable functions \( u \in W^{1,p}(\Omega) \), the following holds:
	
	\[
	\| u - u_\Omega \|_{L^q(\Omega)} \leq C \| \nabla u \|_{L^p(\Omega)}
	\]
	where
	 \( u_\Omega = \frac{1}{|\Omega|} \int_\Omega u \, {\rm d}x \) is the average value of \( u \) over \( \Omega \), $1\leq q\leq \frac{np}{n-p}$.
\end{lem}

The following lemma is an  analogue of the maximum regularity estimate of the heat operator, called as the Solonnikov coercive estimate which will be often used in the proof of our main results.
\begin{lem}[Solonnikov coercive estimates \cite{solonnikov estimates}]\label{Solonnikove estimates}
	Assume $1 < s,l < \infty$, then for every $f \in L^l(0,T;L^s(\mathbb R^3))$, there exists a unique solution $(u, \nabla p)$ of the Stokes system
	\begin{equation*}
		\begin{cases}
			\partial_t u - \Delta u + \nabla p = f, \\
			\operatorname{div} \, u = 0, \\
			u(x,0) = 0
		\end{cases}
	\end{equation*}
	in $ \mathbb{R}^3\times(0,T)$ satisfying
	\begin{equation*}
		\|\partial_t u\|_{L^l(0,T;L^s(\mathbb R^3))}+\|\nabla^2 u\|_{L^l(0,T;L^s(\mathbb R^3))}+\|\nabla p\|_{L^l(0,T;L^s(\mathbb R^3))}\leq C(l,s)\|f\|_{L^l(0,T;L^s(\mathbb R^3))}.
	\end{equation*}
\end{lem}

Finally, we conclude this subsection by giving the following lemma which  plays   a fundamental role in estimating the nonlinear term. Its proof is a simple consequence of the H\"{o}lder and Sobolev inequalities, we omit it here.

\begin{lem}\label{solonnikov estimates-1}
For $u\in L^\infty(0,T;L^2(\mathbb R^3))\cap L^2(0,T;\dot{H}^1(\mathbb R^3))$, we have
\begin{equation*}\label{H_2v_2}
  \|u \nabla u\|_{L^l(0,T;L^s(\mathbb R^3))}\leq C(s,l)(\|u\|_{L^\infty(0,T;L^2(\mathbb R^3))}+\|\nabla u\|_{L^2(0,T;L^2(\mathbb R^3))}),\quad \frac{3}{s}+\frac{2}{l}=4.
\end{equation*}
\end{lem}
\medskip

\section{Existence of \crm{global} weak $L^3$-solutions}
\subsection{A priori \crm{estimates} of \crm{global} weak $L^3$-solutions}
To derive the existence of \crm{global} weak $L^3$-solutions to the system \eqref{EQ-tlh}-\eqref{initial value}, the key point is to \crm{establish a priori estimate} of \crm{global} weak $L^3$-solutions, which is the main purpose of this subsection. Let us state the main result of this subsection as follows.

\begin{lem}\label{Energy Estimate}
 \crm{Let the pair $(v,H)$ be a global weak $L^3$-solutions} to problem \eqref{EQ-tlh}-\eqref{initial value} in the sense of Definition \ref{def1}. Then, for any $T\in(0,\infty)$, we have
\begin{align*}
\|(v_2,H_2)\|_{L^\infty(0,T;L^2(\mathbb R^3))}+\|(\nabla v_2,\nabla H_2)\|_{L^2(0,T;L^2(\mathbb R^3))}
\leq CT^{\frac{3}{4}},
\end{align*}
where the pair \crm{$(v_2,H_2)=(v-e^{t\Delta}v_0, H-e^{t\Delta}h_0)$} is defined in Definition \ref{def1}, and $C$ is a fixed constant depending only on $\|v_0\|_{L^3(\mathbb R^3)}, \|h_0\|_{L^3(\mathbb R^3)}$.
\end{lem}

\begin{proof}
To begin with, by using the H\"{o}lder's inequality, Young's inequality and the global energy inequality \eqref{global energy inequality}, one has
\begin{align*}
&\|v_2(\cdot,t)\|^2_{L^2(\mathbb R^3)}+\|H_2(\cdot,t)\|^2_{L^2(\mathbb R^3)}+2\int_0^t(\|\nabla v_2\|^2_{L^2(\mathbb R^3)}+\|\nabla H_2\|^2_{L^2(\mathbb R^3)})\,{\rm d}s\\
&\leq\int_0^t\int_{\mathbb R^3}\Big(v_1\otimes v_2-H_1\otimes H_2+v_1\otimes v_1-H_1\otimes H_1\Big):\nabla v_2\,{\rm d}x{\rm d}s\\
&\quad+\int_0^t\int_{\mathbb R^3}\Big(H_1\otimes v_2-v_1\otimes H_2+H_1\otimes v_1-v_1\otimes H_1\Big):\nabla H_2\,{\rm d}x{\rm d}s\\
&\leq \int_0^t\|v_1\|_{L^5(\mathbb R^3)}\|v_2\|_{L^\frac{10}{3}(\mathbb R^3)}\|\nabla v_2\|_{L^2(\mathbb R^3)}\,{\rm d}s
+\int_0^t\|H_1\|_{L^5(\mathbb R^3)}\|H_2\|_{L^\frac{10}{3}(\mathbb R^3)}\|\nabla v_2\|_{L^2(\mathbb R^3)}\,{\rm d}s\\
&\quad+\int_0^t\|v_1\|^2_{L^4(\mathbb R^3)}\|\nabla v_2\|_{L^2(\mathbb R^3)}\,{\rm d}s
+\int_0^t\|H_1\|^2_{L^4(\mathbb R^3)}\|\nabla v_2\|_{L^2(\mathbb R^3)}\,{\rm d}s\\
&\quad+\int_0^t\|v_1\|_{L^5(\mathbb R^3)}\|H_2\|_{L^\frac{10}{3}(\mathbb R^3)}\|\nabla H_2\|_{L^2(\mathbb R^3)}\,{\rm d}s
+\int_0^t\|H_1\|_{L^5(\mathbb R^3)}\|v_2\|_{L^\frac{10}{3}(\mathbb R^3)}\|\nabla H_2\|_{L^2(\mathbb R^3)}\,{\rm d}s\\
&\quad+2\int_0^t\|H_1\|_{L^4(\mathbb R^3)}\|v_1\|_{L^4(\mathbb R^3)}\|\nabla H_2\|_{L^2(\mathbb R^3)}\,{\rm d}s\\
&\leq C\int_0^t\|v_1\|^2_{L^5(\mathbb R^3)}\|v_2\|^2_{L^\frac{10}{3}(\mathbb R^3)}\,{\rm d}s
+C\int_0^t\|H_1\|^2_{L^5(\mathbb R^3)}\|H_2\|^2_{L^\frac{10}{3}(\mathbb R^3)}\,{\rm d}s\\
&\quad+C\int_0^t\|v_1\|^4_{L^4(\mathbb R^3)}+c\|H_1\|^4_{L^4(\mathbb R^3)}\,{\rm d}s
+C\int_0^t\|v_1\|^2_{L^5(\mathbb R^3)}\|H_2\|^2_{L^\frac{10}{3}(\mathbb R^3)}\,{\rm d}s\\
&\quad+C\int_0^t\|H_1\|^2_{L^5(\mathbb R^3)}\|v_2\|^2_{L^\frac{10}{3}(\mathbb R^3)}\,{\rm d}s
+C\int_0^t\|H_1\|^2_{L^4(\mathbb R^3)}\|v_1\|^2_{L^4(\mathbb R^3)}\,{\rm d}s.
\end{align*}
Utilizing the interpolation inequality for $L^p$-norms
\begin{align*}
\|(v_1,H_1)\|_{L^4(\mathbb R^3)}
\leq\|(v_1,H_1)\|^{\frac{3}{8}}_{L^3(\mathbb R^3)}\|(v_1,H_1)\|^{\frac{5}{8}}_{L^5(\mathbb R^3)},
\end{align*}
\begin{align*}
\|(v_2,H_2)\|_{L^\frac{10}{3}(\mathbb R^3)}
\leq\|(v_2,H_2)\|^{\frac{2}{5}}_{L^2(\mathbb R^3)}\|(\nabla v_2,\nabla H_2\|^{\frac{3}{5}}_{L^2(\mathbb R^3)},
\end{align*}
and the Young's inequality, we infer that
\begin{align}\label{EE-1}
\begin{split}
&\|v_2(\cdot,t)\|^2_{L^2(\mathbb R^3)}+\|H_2(\cdot,t)\|^2_{L^2(\mathbb R^3)}+2\int_0^t(\|\nabla v_2\|^2_{L^2(\mathbb R^3)}+\|\nabla H_2\|^2_{L^2(\mathbb R^3)})\,{\rm d}s\\
&\leq C\int_0^t(\|v_1\|^5_{L^5(\mathbb R^3)}+\|H_1\|^5_{L^5(\mathbb R^3)})(\|v_2\|^2_{L^2(\mathbb R^3)}+\|H_2\|^2_{L^2(\mathbb R^3)})\,{\rm d}s\\
&+CT^{\frac{1}{2}}\Big(\|v_1\|^{\frac{3}{2}}_{L^{\infty}(0,T;L^3(\mathbb R^3))}\|v_1\|^{\frac{5}{2}}_{L^{5}(0,T;L^5(\mathbb R^3))}+\|H_1\|^{\frac{3}{2}}_{L^{\infty}(0,T;L^3(\mathbb R^3))}\|H_1\|^{\frac{5}{2}}_{L^{5}(0,T;L^5(\mathbb R^3))}\\
&+\|v_1\|^{\frac{3}{4}}_{L^{\infty}(0,T;L^3(\mathbb R^3))}\|H_1\|^{\frac{3}{4}}_{L^{\infty}(0,T;L^3(\mathbb R^3))}\|v_1\|^{\frac{5}{4}}_{L^{5}(0,T;L^5(\mathbb R^3))}\|H_1\|^{\frac{5}{4}}_{L^{5}(0,T;L^5(\mathbb R^3))}\Big)\\&
\leq C\int_0^t(\|v_1\|^5_{L^5(\mathbb R^3)}+\|H_1\|^5_{L^5(\mathbb R^3)})(\|v_2\|^2_{L^2(\mathbb R^3)}+\|H_2\|^2_{L^2(\mathbb R^3)})\,{\rm d}s\\
&+CT^{\frac{1}{2}}\left(\|v_0\|_{L^3(\mathbb R^3)}^4+\|h_0\|_{L^3(\mathbb R^3)}^4\right)
\end{split}
\end{align}
From the Lemma \ref{Gronwall inequality} and \eqref{v_1}, one has for any $t\in(0,T)$
\crm{\begin{align*}
\|v_2(\cdot,t)\|^2_{L^2(\mathbb R^3)}+\|H_2(\cdot,t)\|^2_{L^2(\mathbb R^3)}&\leq C\|(v_0,h_0)\|^4_{L^3(\mathbb R^3)}T^{\frac{3}{2}}
\exp\Big(\int_0^T\big(\|v_1\|^5_{L^5(\mathbb R^3)}+\|H_1\|^5_{L^5(\mathbb R^3)}\,{\rm d}s\big)\Big)\\
&\leq CT^{\frac{3}{2}}
\end{align*}
with some constant $C$  depending only on $\|v_0\|_{L^3(\mathbb R^3)}, \|h_0\|_{L^3(\mathbb R^3)}$.}
From this and inequality \eqref{EE-1}, it is clear that
\begin{align*}
\int_0^T\left(\|\nabla v_2(\cdot,s)\|^2_{L^2(\mathbb R^3)}+\|\nabla H_2(\cdot,s)\|^2_{L^2(\mathbb R^3)}\right)\,{\rm d}s\leq C(\|v_0\|_{L^3(\mathbb R^3)}, \|h_0\|_{L^3(\mathbb R^3)})T^{\frac{1}{2}},
\end{align*}
\crm{as required.}
\end{proof}
\medskip

\crm{As a by-product of this lemma}, we further obtain by invoking Lemma \ref{Solonnikove estimates} and Lemma \ref{solonnikov estimates-1}:
\begin{prop}\label{pressure estimate}
Let $(v_2,H_2,\Pi)$ be defined in Definition \ref{def1}, and it can be decomposed as follows
\begin{equation*}
\begin{aligned}
(v_2,H_2,\Pi)=\sum_{i=1}^{4}(v^i_2,H^i_2,\Pi^i).
\end{aligned}
\end{equation*}
Here, $(v^i_2,H^i_2,\Pi^i)~(i=1,2,3,4)$ \crm{satisfies} the following equations in $\mathbb{R}^{3}\times(0,\infty)$ respectively
\begin{equation*}
\left\{
\begin{array}{l}
\begin{aligned}
&\partial_t v^1_2-\Delta v^1_2+\nabla \Pi^1=H_2\nabla H_2-v_2\nabla v_2,\\
&\partial_t H^1_2-\Delta H^1_2=H_2\nabla v_2-v_2\nabla H_2,\\
&v^1_2(x,0)=H^1_2(x,0)=0,\\
&\div v^1_2=\div H^1_2=0,
\end{aligned}
\end{array} \right.
\end{equation*}
\begin{equation*}
\left\{
\begin{array}{l}
\begin{aligned}
&\partial_t v^2_2-\Delta v^2_2+\nabla \Pi^2=H_2\nabla H_1-v_2\nabla v_1,\\
&\partial_t H^2_2-\Delta H^2_2=H_2\nabla v_1-v_2\nabla H_1,\\
&v^2_2(x,0)=H^2_2(x,0)=0,\\
&\div v^2_2=\div H^2_2=0,
\end{aligned}
\end{array} \right.
\end{equation*}
\begin{equation*}
\left\{
\begin{array}{l}
\begin{aligned}
&\partial_t v^3_2-\Delta v^3_2+\nabla \Pi^3=H_1\nabla H_2-v_1\nabla v_2,\\
&\partial_t H^3_2-\Delta H^3_2=H_1\nabla v_2-v_1\nabla H_2,\\
&v^3_2(x,0)=H^3_2(x,0)=0,\\
&\div v^3_2=\div H^3_2=0,
\end{aligned}
\end{array} \right.
\end{equation*}
\begin{equation*}
\left\{
\begin{array}{l}
\begin{aligned}
&\partial_t v^4_2-\Delta v^4_2+\nabla \Pi^4=H_1\nabla H_1-v_1\nabla v_1,\\
&\partial_t H^4_2-\Delta H^4_2=H_1\nabla v_1-v_1\nabla H_1,\\
&v^4_2(x,0)=H^4_2(x,0)=0,\\
&\div v^4_2=\div H^4_2=0.
\end{aligned}
\end{array} \right.
\end{equation*}
 Then for any $T\in (0,+\infty)$, we have
\begin{equation*}
\begin{aligned}
&(v^1_2,H^1_2)\in W^{1,\frac{3}{2}}(0,T;W^{2,\frac{9}{8}}(\mathbb R^3)),\quad \nabla\Pi^1\in L^{\frac{3}{2}}(0,T;L^{\frac{9}{8}}(\mathbb R^3)),\\
&(v^2_2,H^2_2)\in W^{1,\frac{3}{2}}(0,T;W^{2,\frac{4}{3}}(\mathbb R^3)),\quad \nabla\Pi^2\in L^{\frac{3}{2}}(0,T;L^{\frac{4}{3}}(\mathbb R^3)),\\
&(v^3_2,H^3_2)\in W^{1,\frac{3}{2}}(0,T;W^{2,\frac{6}{5}}(\mathbb R^3)),\quad \nabla\Pi^3\in L^{\frac{3}{2}}(0,T;L^{\frac{6}{5}}(\mathbb R^3)),\\
&(v^4_2,H^4_2)\in W^{1,\frac{3}{2}}(0,T;W^{2,\frac{3}{2}}(\mathbb R^3)),\quad \nabla\Pi^4\in L^{\frac{3}{2}}(0,T;L^{\frac{3}{2}}(\mathbb R^3)).
\end{aligned}
\end{equation*}
Moreover, $\Pi^1,\Pi^2, \Pi^3,\Pi^4$ fulfill the following Poincare type inequality respectively:
\begin{equation}\label{Pi}
\begin{aligned}
&\int_{0}^{T}\int_{B(x_0,R)}\Big|\Pi^1-[\Pi^1]_{B(x_0,R)}\Big|^{\frac{3}{2}}\,{\rm d}x{\rm d}t\leq CR^\frac{1}{2}\int_{0}^{T}\Big(\int_{B(x_0,R)}|\nabla\Pi^1|^\frac{9}{8}dx\Big)^\frac{4}{3}\,{\rm d}t,\\		&\int_{0}^{T}\int_{B(x_0,R)}\Big|\Pi^2-[\Pi^2]_{B(x_0,R)}\Big|^{\frac{3}{2}}\,{\rm d}x{\rm d}t\leq CR^\frac{9}{8}\int_{0}^{T}\Big(\int_{B(x_0,R)}|\nabla\Pi^2|^\frac{4}{3}dx\Big)^\frac{9}{8}\,{\rm d}t,\\		&\int_{0}^{T}\int_{B(x_0,R)}\Big|\Pi^3-[\Pi^3]_{B(x_0,R)}\Big|^{\frac{3}{2}}\,{\rm d}x{\rm d}t\leq CR^\frac{3}{4}\int_{0}^{T}\Big(\int_{B(x_0,R)}|\nabla\Pi^3|^\frac{6}{5}dx\Big)^\frac{5}{4}\,{\rm d}t,\\
				&\int_{0}^{T}\int_{B(x_0,R)}\Big|\Pi^4-[\Pi^4]_{B(x_0,R)}\Big|^{\frac{3}{2}}\,{\rm d}x{\rm d}t\leq CR^\frac{3}{2}\int_{0}^{T}\int_{B(x_0,R)}|\nabla\Pi^4|^\frac{3}{2}\,{\rm d}x{\rm d}t.
\end{aligned}
\end{equation}
Here, $[\Pi^i]_{B(x_0,R)}~(i=1,2,3,4)$ denotes the average of $\Pi^i$ over $B(x_0,R)$, $C$ is \crm{a constant depending only on $\|v_0\|_{L^3(\mathbb R^3)}, \|h_0\|_{L^3(\mathbb R^3)},T$}.
\end{prop}

\begin{proof}
 The main tool to prove this proposition is Lemma \ref{Solonnikove estimates}. For $v^1_2$, \crm{from the Lemma \ref{Energy Estimate}, Lemma \ref{solonnikov estimates-1} and Lemma \ref{Solonnikove estimates}}, one has
			\begin{equation*}
				\begin{aligned}
					&\|\partial_t v^1_2\|_{L^\frac{3}{2}(0,T;L^\frac{9}{8}(\mathbb R^3))}+\|\nabla^2 v^1_2\|_{L^\frac{3}{2}(0,T;L^\frac{9}{8}(\mathbb R^3))}+\|\nabla\Pi^1\|_{L^\frac{3}{2}(0,T;L^\frac{9}{8}(\mathbb R^3))}\\
					&\leq C\|H_2\nabla H_2-v_2\nabla v_2\|_{L^\frac{3}{2}(0,T;L^\frac{9}{8}(\mathbb R^3))}\\
&\crm{\leq C\|H_2\|_{L^{\infty}(0,T;L^2(\R^3))}\|\nabla H_2\|^{2/3}_{L^{2}(0,T;L^2(\R^3))}}\\&
\crm{\leq CT^{5/4}.}
				\end{aligned}
			\end{equation*}
Now we turn to the estimate of $v^2_2$. In fact, by the known estimate of the heat potential, one has
\begin{align*}
\begin{split}
\|\nabla v_1\|_{L^4(\mathbb R^3)}\leq Ct^{-\frac{5}{8}}\|v_0\|_{L^3(\mathbb R^3)},\quad
\|\nabla H_1\|_{L^4(\mathbb R^3)}\leq Ct^{-\frac{5}{8}}\|h_0\|_{L^3(\mathbb R^3)}.
\end{split}
\end{align*}
This, along with the Lemma \ref{Energy Estimate}, H\"{o}lder's inequality, yields
\begin{equation*}
\begin{aligned}
&\|H_2\nabla H_1-v_2\nabla v_1\|_{L^\frac{3}{2}(0,T;L^\frac{4}{3}(\mathbb R^3))}\\
&\leq \|v_2\|_{L^\infty(0,T;L^2(\mathbb R^3))}
 \|\nabla v_1\|_{L^\frac{3}{2}(0,T;L^4(\mathbb R^3))}
 +\|H_2\|_{L^\infty(0,T;L^2(\mathbb R^3))}
 \|\nabla H_1\|_{L^\frac{3}{2}(0,T;L^4(\mathbb R^3))}\\
&\leq C\|v_2\|_{L^\infty(0,T;L^2(\mathbb R^3))}\Big\|t^{-\frac{5}{8}}\|v_0\|_{L^3(\mathbb R^3)}\Big\|_{L^\frac{3}{2}(0,T)} +C\|H_2\|_{L^\infty(0,T;L^2(\mathbb R^3))}\Big\|t^{-\frac{5}{8}}\|H_0\|_{L^3(\mathbb R^3)}\Big\|_{L^\frac{3}{2}(0,T)}\\
&\leq C\|v_2\|_{L^\infty(0,T;L^2(\mathbb R^3))}T^\frac{1}{24}\|v_0\|_{L^3(\mathbb R^3)}
+C\|H_2\|_{L^\infty(0,T;L^2(\mathbb R^3))}T^\frac{1}{24}\|h_0\|_{L^3(\mathbb R^3)}\\
&\leq C(\|v_0\|_{L^3(\mathbb R^3)},\|h_0\|_{L^3(\mathbb R^3)})T^\frac{7}{24},
\end{aligned}
\end{equation*}
which further implies by Lemma \ref{Solonnikove estimates}
\begin{equation*}
\begin{aligned}
&\|\partial_t v^2_2\|_{L^\frac{3}{2}(0,T;L^\frac{4}{3}(\mathbb R^3))}+\|\nabla^2 v^2_2\|_{L^\frac{3}{2}(0,T;L^\frac{4}{3}(\mathbb R^3))}+\|\nabla\Pi^2\|_{L^\frac{3}{2}(0,T;L^\frac{4}{3}(\mathbb R^3))}\\
&\leq C\|H_2\nabla H_1-v_2\nabla v_1\|_{L^\frac{3}{2}(0,T;L^\frac{4}{3}(\mathbb R^3))}\\
&\leq C(\|v_0\|_{L^3(\mathbb R^3)},\|h_0\|_{L^3(\mathbb R^3)})T^\frac{7}{24}.
\end{aligned}
\end{equation*}
For $v^3_2$, by using the Lemma \ref{Energy Estimate} and H\"{o}lder's inequality, one has
\begin{equation*}
\begin{aligned}
&\|H_1\nabla H_2-v_1\nabla v_2\|_{L^{\frac{3}{2}}(0,T;L^{\frac{6}{5}}(\mathbb R^3))}\\&
\leq \|v_1\|_{L^\infty(0,T;L^3(\mathbb R^3))}\|\nabla v_2\|_{L^{\frac{3}{2}}(0,T;L^2(\mathbb R^3))}+\|H_1\|_{L^\infty(0,T;L^3(\mathbb R^3))}|\nabla H_2\|_{L^{\frac{3}{2}}(0,T;L^2(\mathbb R^3))}\\
&\leq \|v_1\|_{L^\infty(0,T;L^3(\mathbb R^3))}\|\nabla v_2\|_{L^2(0,T;L^2(\mathbb R^3))}T^\frac{1}{6}
+\|H_1\|_{L^\infty(0,T;L^3(\mathbb R^3))}|\nabla H_2\|_{L^2(0,T;L^2(\mathbb R^3))}T^\frac{1}{6}\\
&\leq C(\|v_0\|_{L^3(\mathbb R^3)},\|h_0\|_{L^3(\mathbb R^3)})T^\frac{2}{3}.
\end{aligned}
\end{equation*}
This, along with Lemma \ref{Solonnikove estimates}, derives,
			\begin{equation*}
				\begin{aligned}
					&\|\partial_t v^3_2\|_{L^{\frac{3}{2}}(0,T;L^{\frac{6}{5}}(\mathbb R^3))}+\|\nabla^2 v^3_2\|_{L^{\frac{3}{2}}(0,T;L^{\frac{6}{5}}(\mathbb R^3))}+\|\nabla\Pi^3\|_{L^{\frac{3}{2}}(0,T;L^{\frac{6}{5}}(\mathbb R^3))}\\
					&\leq C(\|v_0\|_{L^3(\mathbb R^3)},\|h_0\|_{L^3(\mathbb R^3)})T^\frac{2}{3}.
				\end{aligned}
			\end{equation*}
For $v^4_2$, by using H\"{o}lder's inequality and \eqref{v_1}, one has
\begin{equation*}
\begin{aligned}
&\|H_1\nabla H_1-v_1\nabla v_1\|_{L^\frac{3}{2}(0,T;L^\frac{3}{2}(\mathbb R^3))}\\
&\leq \|v_1\nabla v_1\|_{L^\frac{3}{2}(0,T;L^\frac{3}{2}(\mathbb R^3))}
+\|H_1\nabla H_1\|_{L^\frac{3}{2}(0,T;L^\frac{3}{2}(\mathbb R^3))}\\
&\leq \|v_1\|_{L^\infty(0,T;L^3(\mathbb R^3))}\|\nabla v_1\|_{L^\frac{3}{2}(0,T;L^3(\mathbb R^3))}
+\|H_1\|_{L^\infty(0,T;L^3(\mathbb R^3))}\|\nabla H_1\|_{L^\frac{3}{2}(0,T;L^3(\mathbb R^3))}\\
&\leq C\|v_0\|^2_{L^3(\mathbb R^3)}\|t^{-\frac{1}{2}}\|_{L^\frac{3}{2}(0,T)}+c\|h_0\|^2_{L^3(\mathbb R^3)})\|t^{-\frac{1}{2}}\|_{L^\frac{3}{2}(0,T)}\\
&\leq C(\|v_0\|^2_{L^3(\mathbb R^3)}+\|h_0\|^2_{L^3(\mathbb R^3)})T^\frac{1}{6},
\end{aligned}
\end{equation*}
which implies,
			\begin{equation*}
				\begin{aligned}
					&\|\partial_t v^4_2\|_{L^\frac{3}{2}(0,T;L^\frac{3}{2}(\mathbb R^3))}+\|\nabla^2 v^4_2\|_{L^\frac{3}{2}(0,T;L^\frac{3}{2}(\mathbb R^3))}+\|\nabla\Pi^4\|_{L^\frac{3}{2}(0,T;L^\frac{3}{2}(\mathbb R^3))}\\
					&\leq C\|H_1\nabla H_1-v_1\nabla v_1\|_{L^\frac{3}{2}(0,T;L^\frac{3}{2}(\mathbb R^3))}\leq C(\|v_0\|^2_{L^3(\mathbb R^3)}+\|h_0\|^2_{L^3(\mathbb R^3)})T^\frac{1}{6}.
				\end{aligned}
			\end{equation*}
		Combing the above estimates and the Lemma \ref{sobolev-poincare inequality}, we immediately derive \eqref{Pi}.
		\end{proof}

\subsection{The existence of approximate solutions}
	In this subsection, we are going to prove the existence of \crm{global} weak $L^3$-solutions of \eqref{EQ-tlh}-\eqref{initial value} which  is equivalent to  gain  the existence of energy weak solutions (usually called as Leray-Hopf weak solution) for \eqref{v_2}.
	To achieve this goal,  we first smooth the nonlinear term of \eqref{v_2} to obtain its approximate system
\begin{equation}\label{smooth equation}
	\left\{
	\begin{array}{l}
		\begin{aligned}
			&\partial_{t}v_2-\Delta v_2+\nabla \Pi+(\eta_{\epsilon}\ast v_2)\cdot\nabla v_2+v_2\cdot\nabla v_1+ v_1\cdot\nabla v_2-(\eta_{\epsilon}\ast H_2)\cdot\nabla H_2\\
			&\quad-H_2\cdot\nabla H_1- H_1 \cdot\nabla H_2+ v_1\cdot\nabla v_1-H_1\cdot\nabla H_1=0,\\
			&\partial_{t}H_2-\Delta H_2+(\eta_{\epsilon}\ast v_2)\cdot\nabla H_2+ v_2\cdot\nabla H_1+v_1\cdot\nabla H_2-(\eta_{\epsilon}\ast H_2)\cdot\nabla v_2\\
			&\quad-H_2\cdot\nabla v_1-H_1\cdot\nabla v_2+ v_1\cdot\nabla H_1-H_1\cdot\nabla v_1=0,\\
&v_2(x,0)=H_2(x,0)=0,
		\end{aligned}
	\end{array} \right.
\end{equation}
\crm{where $\eta_\epsilon=\epsilon^{-3}\eta\left(\frac{x}{\epsilon}\right)$ is a standard spatial mollifier with $\eta\in C_{0}^{\infty}(\R^{3})$ and $\int_{\R^3}\eta dx=1$.  We then seek} the solution pair $(v_{2\epsilon},H_{2\epsilon})$, usually called the approximate solution,  to the above approximate system by using Lemma \ref{fixed point thm}.
Subsequently, the standard energy method shows that the solution pair $(v_{2\epsilon},H_{2\epsilon})$ is bounded in $L^\infty(0,T;L^2(\mathbb R^3))\cap L^2(0,T;\dot{H}^1(\mathbb R^3))$ uniformly in $\epsilon$. For simplicity, we write the Banach space $X$  as
\begin{equation*}
	\begin{aligned}
		X_{T}= L^\infty(0,T;L^2(\mathbb R^3))\cap L^2(0,T;\dot{H}^1(\mathbb R^3))
	\end{aligned}
\end{equation*}
and define the norm as
\begin{equation*}
	\begin{aligned}
		\|(v,H)\|_{X_T}=\|v\|_{X_{T}}+\|H\|_{X_{T}}.
	\end{aligned}
	\end{equation*}
\crm{Here and in what follows, we always suppose $T<\infty$.}

\medskip

By applying Lemma \ref{fixed point thm}, \crm{we are now} in a position to   construct  a solution to  the approximate system \eqref{smooth equation}.
\begin{thm}\label{TH3.3}
For any $\epsilon>0$, the approximate system \eqref{smooth equation} admits a solution pair $(v_{2\epsilon},H_{2\epsilon})\in X_{T}$ for any $T\in(0,\infty)$ such that
\begin{align}\label{1-4-1}
\|(v_{2\epsilon},H_{2\epsilon})\|^2_{L^\infty(0,T;L^2(\mathbb R^3)) }+\|(v_{2\epsilon},H_{2\epsilon})\|^2_{L^2(0,T;\dot{H}^{1}(\mathbb R^3))}\leq C(\|v_0\|_{L^3(\mathbb R^3)},\|h_0\|_{L^3(\mathbb R^3)})\sqrt{T}.
\end{align}
\end{thm}
\begin{proof}
The proof consists of two steps:  \crm{In step 1, we construct a energy weak solution $(v_{2\epsilon},H_{2\epsilon})\in X_T$} for the approximate system \eqref{smooth equation} in $\R^3\times(0,T)$ with small $T$; subsequently, in step 2 the standard energy estimate shows the energy weak solution constructed in step 1 can be extended in $(0,T)$ for any large $T\in (0,\infty)$.

\textbf{Step 1:}~ {\em The approximate system \eqref{smooth equation} admits a energy weak solution $\crm{(v_{2\epsilon},H_{2\epsilon})}\in X_T$ with small $T$.}
To achieve this, we define the bilinear operator in $X_T$
$$
\mathcal{M}\left((\tilde{v}_1,\tilde{H}_1),(\tilde{v}_2,\tilde{H}_2)\right)
=\left(\mathcal{M}_1\left((\tilde{v}_1,\tilde{H}_1),(\tilde{v}_2,\tilde{H}_2)\right),
\mathcal{M}_2\left((\tilde{v}_1,\tilde{H}_1),(\tilde{v}_2,\tilde{H}_2)\right)\right)
$$
for any $(\tilde{v}_1,\tilde{H}_1),(\tilde{v}_2,\tilde{H}_2) \in X_T$. Here, $\mathcal{M}_1$ and $\mathcal{M}_2$ are respectively given by
\begin{align*} &\mathcal{M}_1\Big((\tilde{v}_1,\tilde{H}_1),(\tilde{v}_2,\tilde{H}_2)\Big)\\
&=\int_{0}^{t}W_{(t-s)}\ast \mathbb P\Big((\eta_{\epsilon}\ast\tilde{v}_1)\nabla \tilde{v}_2+\tilde{v}_1\nabla v_1+v_1\nabla\tilde{v}_1
-(\eta_{\epsilon}\ast\tilde{H}_1)\nabla\tilde{H}_2\\
&\quad-\tilde{H}_1\nabla H_1-H_1\nabla\tilde{H}_1\Big)\,{\rm d}s
+\int_{0}^{t}W_{(t-s)}\ast \mathbb P\Big(v_1\nabla v_1- H_1\nabla H_1\Big)\,{\rm d}s\\
&\mathcal{M}_2\Big((\tilde{v}_1,\tilde{H}_1),(\tilde{v}_2,\tilde{H}_2)\Big)\\
&=\int_{0}^{t}W_{(t-s)}\ast \mathbb P\Big((\eta_{\epsilon}\ast\tilde{v}_1)\nabla\tilde{H}_2+\tilde{v}_1\nabla H_1+v_1\nabla\tilde{H}_1-(\eta_{\epsilon}\ast\tilde{H}_1)\nabla\tilde{v}_2\\
&\quad-\tilde{H}_1\nabla v_1-H_1\nabla\tilde{v}_1\Big)\,{\rm d}s
+\int_{0}^{t}W_{(t-s)}\ast \mathbb P\Big(v_1\nabla H_1-H_1\nabla v_1\Big)\,{\rm d}s.
\end{align*}
\crm{Take note} that proving the existence of a solution to the system \eqref{smooth equation} is equivalent to proving the existence of a fixed point for the integral equations
$$
(\tilde{v}_2,\tilde{H}_2)=\mathcal{M}\Big((\tilde{v}_2,\tilde{H}_2),(\tilde{v}_2,\tilde{H}_2)\Big).
$$
To do this,  we need to show by Lemma \ref{fixed point thm}
\begin{equation*}	
\Big\|\mathcal{M}\Big((\tilde{v}_1,\tilde{H}_1),(\tilde{v}_2,\tilde{H}_2)\Big)\Big\|_{X_T}
\leq c_1\|(\tilde{v}_1,\tilde{H}_1)\|_{X_T}\|(\tilde{v}_2,\tilde{H}_2)\|_{X_T}
+c_2\|(\tilde{v}_1,\tilde{H}_1)\|_{X_T}+\|\mathcal{R}\|_{X_T},\\
\end{equation*}
\begin{equation*}
4c_1\|\mathcal{R}\|_{X_T}<(1-c_2)^2,
\end{equation*}
for any $(\tilde{v}_1,\tilde{H}_1),(\tilde{v}_2,\tilde{H}_2)\in X_T$. To verify the above conditions, we now estimate each term of $\mathcal{M}\Big((\tilde{v}_1,\tilde{H}_1),(\tilde{v}_2,\tilde{H}_2)\Big)$. For $\int_{0}^{t}W_{(t-s)}\ast \mathbb P(\eta_{\epsilon}\ast\tilde{v}_1)\nabla\tilde{v}_2\,{\rm d}s$, by the H\"{o}lder's inequality and classical properties of mollifiers, one has
\begin{equation*}
	\begin{aligned}
		\|\eta_{\epsilon}\ast\tilde{v}_1(\cdot,t)\|_{L^\infty(\mathbb R^3)}&\leq \|\eta_{\epsilon}\|_{L^2(\mathbb R^3)}\|\tilde{v}_1(\cdot,t)\|_{L^2(\mathbb R^3)}\\
		&=\Big(\int_{\mathbb{R}^3}\Big|\frac{1}{\epsilon^3}\eta\Big(\frac{x}{\epsilon}\Big)\Big|^2{\rm d}x\Big)^{\frac{1}{2}}\|\tilde{v}_1\|_{L^2(\mathbb R^3)}\\
	    &=\epsilon^{-\frac{3}{2}}\|\tilde{v}_1(\cdot,t)\|_{L^2(\mathbb R^3)}\|\eta\|_{L^2(\mathbb R^3)}.
	\end{aligned}
\end{equation*}
This, together with the Lemma \ref{H^-1 to X_T}, yields,
\begin{equation*}
	\begin{aligned}
		\Big\|\int_{0}^{t}W_{(t-s)}\ast \mathbb P(\eta_{\epsilon}\ast\tilde{v}_1)\nabla\tilde{v}_2\,{\rm d}s\Big\|_{X_T}&\leq \|(\eta_{\epsilon}\ast\tilde{v}_1)\nabla\tilde{v}_2\|_{L^2((0,T),\dot{H}^{-1}(\mathbb R^3))}\\
		&\leq C\|(\eta_{\epsilon}\ast\tilde{v}_1)\otimes\tilde{v}_2\|_{L^2(0,T;L^2(\mathbb R^3))}\\
		&\leq C\Big(\int_{0}^{T}\|\eta_{\epsilon}\ast\tilde{v}_1\|^2_{L^\infty(\mathbb R^3)}\|\tilde{v}_2\|^2_{L^2(\mathbb R^3)}\,{\rm d}t\Big)^\frac{1}{2}\\
		&\leq C \sqrt{T}\epsilon^{-\frac{3}{2}}\|\tilde{v}_1\|_{L^\infty(0,T; L^2(\mathbb R^3))}\|\tilde{v}_2\|_{L^\infty(0,T; L^2(\mathbb R^3))}\\
        &\leq C \sqrt{T}\epsilon^{-\frac{3}{2}}\|\tilde{v}_1\|_{X_T}\|\tilde{v}_2\|_{X_T}
	\end{aligned}
\end{equation*}
For $\int_{0}^{t}W_{(t-s)}\ast \mathbb P(\tilde{v}_1\nabla v_1)\, {\rm d}s$, by the H\"{o}lder's inequality and Lemma \ref{H^-1 to X_T}, one has,
\begin{equation*}
	\begin{aligned}
		\Big\|\int_{0}^{t}W_{(t-s)}\ast \mathbb P(\tilde{v}_1\nabla v_1)\,{\rm d}s\Big\|_{X_T}&\leq\|\tilde{v}_1\nabla v_1\|_{(L^2(0,T),\dot{H}^{-1})}\\
		&\leq C\|v_1\otimes \tilde{v}_1\|_{L^2(0,T;L^2(\mathbb R^3))}\\
		&\leq C\|v_1\|_{L^5(0,T;L^5(\mathbb R^3)) }\|\tilde{v}_1\|_{L^{\frac{10}{3}}(0,T;L^{\frac{10}{3}}(\mathbb R^3))}\\
&\leq C\|v_1\|_{L^5(0,T;L^5(\mathbb R^3)) }\|\tilde{v}_1\|_{X_T}.
	\end{aligned}
\end{equation*}
Similarly,
\begin{equation*}
	\begin{aligned}
		\Big\|\int_{0}^{t}W_{(t-s)}\ast \mathbb P(v_1\nabla\tilde{v}_1)\,{\rm d}s\Big\|_{X_T}&\leq \|v_1\nabla\tilde{v}_1\|_{L^2(0,T;\dot{H}^{-1}(\mathbb R^3))}\\
		&\leq C\|v_1\otimes\tilde{v}_1\|_{L^2(0,T;L^2(\mathbb R^3))}\\
		&\leq C\|v_1\|_{L^5(0,T;L^5(\mathbb R^3)) }\|\tilde{v}_1\|_{L^{\frac{10}{3}}(0,T;L^{\frac{10}{3}}(\mathbb R^3))}\\
&\leq C\|v_1\|_{L^5(0,T;L^5(\mathbb R^3))}\|\tilde{v}_1\|_{X_T}
	\end{aligned}
\end{equation*}
For $\int_{0}^{t}W_{(t-s)}\ast \mathbb P(v_1\nabla v_1) \,{\rm d}s$, by the H\"{o}lder's inequality and Lemma \ref{H^-1 to X_T}, one has
\begin{equation*}
	\begin{aligned}
		\Big\|\int_{0}^{t}W_{(t-s)}\ast \mathbb P(v_1\nabla v_1)\,{\rm d}s\Big\|_{X_T}&\leq \|v_1\nabla v_1\|_{L^2(0,T;\dot{H}^{-1}(\mathbb R^3))}\\
		&\leq C\|v_1\otimes v_1\|_{L^2(0,T;L^2(\mathbb R^3))}\\
		&\leq C\|v_1\|_{L^4(0,T;L^4(\mathbb R^3))}^2\\
	\end{aligned}
\end{equation*}
The remaining terms can be estimated through the same process as above. Indeed, summing all of these yields
\begin{equation*}
	\begin{aligned}
\Big\|\mathcal{M}_1\Big((\tilde{v}_1,\tilde{H}_1),(\tilde{v}_2,\tilde{H}_2)\Big)\Big\|_{X_T}&\leq C\sqrt{T}\epsilon^{-\frac{3}{2}}\Big(\|\tilde{v}_1\|_{X_T}\|\tilde{v}_2\|_{X_T} +\|\tilde{H}_1\|_{X_T}\|\tilde{H}_2\|_{X_T}\Big)\\
		&\quad+C\|v_1\|_{L^5(0,T;L^5(\mathbb R^3))}\|\tilde{v}_1\|_{X_T}+C\|H_1\|_{L^5(0,T;L^5(\mathbb R^3))}\|\tilde{H}_1\|_{X_T}\\
		&\quad+C\|v_1\|_{L^4(0,T;L^4(\mathbb R^3))}^2+C\|H_1\|_{L^4(0,T;L^4(\mathbb R^3))}^2,
	\end{aligned}
\end{equation*}
\begin{equation*}
\begin{aligned}
\Big\|\mathcal{M}_2\Big((\tilde{v}_1,\tilde{H}_1),(\tilde{v}_2,\tilde{H}_2)\Big)\Big\|_{X_T}&\leq C\sqrt{T}\epsilon^{-\frac{3}{2}}\Big(\|\tilde{v}_1\|_{X_T}\|\tilde{H}_2\|_{X_T} +\|\tilde{H}_1\|_{X_T}\|\tilde{v}_2\|_{X_T}\Big)\\
		&\quad+C\|v_1\|_{L^5(0,T;L^5(\mathbb R^3))}\|\tilde{H}_1\|_{X_T}
+C\|H_1\|_{L^5(0,T;L^5(\mathbb R^3))}\|\tilde{v}_1\|_{X_T}\\
		&\quad+2C\|v_1\|_{L^4(0,T;L^4(\mathbb R^3))}\|H_1\|_{L^4(0,T;L^4(\mathbb R^3))}.
	\end{aligned}
\end{equation*}
Combing the above relations, we easily derive that
\begin{align*}
\Big\|\mathcal{M}\Big((\tilde{v}_1,\tilde{H}_1),(\tilde{v}_2,\tilde{H}_2)\Big)\Big\|_{X_T}&\leq C\sqrt{T}\epsilon^{-\frac{3}{2}}\Big(\|\tilde{v}_1\|_{X_T}+\|\tilde{H}_1\|_{X_T}\Big)
\Big(\|\tilde{v}_2\|_{X_T}+\|\tilde{H}_2\|_{X_T}\Big)\\
&\quad+C\Big(\|v_1\|_{L^5(0,T;L^5(\mathbb R^3))}+\|H_1\|_{L^5(0,T;L^5(\mathbb R^3))}\Big)\Big(\|\tilde{v}_1\|_{X_T}+\|\tilde{H}_1\|_{X_T}\Big)\\
&\quad+C\Big(\|v_1\|_{L^4(0,T;L^4(\mathbb R^3))}+\|H_1\|_{L^4(0,T;L^4(\mathbb R^3))}\Big)^2,
\end{align*}
which implies
\begin{align*}
\Big\|\mathcal{M}\Big((\tilde{v}_1,\tilde{H}_1),(\tilde{v}_2,\tilde{H}_2)\Big)\Big\|_{X_T}
&\leq c_1
\|(\tilde{v}_1,\tilde{H}_1)\|_{X_T}\|(\tilde{v}_2,\tilde{H}_2)\|_{X_T}+c_2\|(\tilde{v}_1,\tilde{H}_1)\|_{X_T}+\|\mathcal{R}\|_{X_T},
\end{align*}
where
\begin{equation*}
	\begin{aligned}
		&c_1=C\sqrt{T}\epsilon^{-\frac{3}{2}},\quad c_2=C\Big(\|v_1\|_{L^5(0,T;L^5(\mathbb R^3))}+\|H_1\|_{L^5(0,T;L^5(\mathbb R^3))}\Big),\\
&\mathcal{R}(x,t)=\int_{0}^{t}W_{(t-s)}\ast \mathbb P\Big(v_1\nabla v_1- H_1\nabla H_1\Big)\,{\rm d}s+\int_{0}^{t}W_{(t-s)}\ast \mathbb P\Big(v_1\nabla H_1-H_1\nabla v_1\Big)\,{\rm d}s,
	\end{aligned}
\end{equation*}
with
$$
\|\mathcal{R}\|_{X_T}\leq C\|(v_1,H_1)\|^2_{L^4(0,T;L^4(\mathbb R^3))}\leq  C\|(v_1,H_1)\|^2_{L^4(0,1;L^4(\mathbb R^3))}.
$$
Now we take $T_1\ll1$ small such that $c_2<1$, and derive by a simple calculation
$$
4c_1\|\mathcal{R}\|_{X_T}<(1-c_2)^2
$$
with
\begin{align*}
T=\min\Big\{T_1,\frac{1}{16C^{4}}\epsilon^{3}\Big(\|(v_1,H_1)\|_{L^4(0,1;L^4(\mathbb R^3))}\Big)^{-4}\Big(1-C\Big(\|(v_1,H_1)\|_{L^5(0,T_1;L^5(\mathbb R^3))}\Big)\Big)^4\Big\}.
\end{align*}
This, along with Lemma \ref{fixed point thm}, shows that  \eqref{smooth equation} admits a solution, denoted by $(v_{2\epsilon},H_{2\epsilon})$ which belongs to $X_T$ and satisfies
$$\|(v_{2\epsilon},H_{2\epsilon})\|_{X_T}\leq \frac{1-c_2-\sqrt{(c_2-1)^2-4c_1\|\mathcal{R}\|_{X_T}}}{2c_1}.
$$
Finally, from the previous proof process, it is known that $$(f_1,f_2)\in L^2(0,T;{H}^{-1}(\mathbb R^3)),$$
where
\begin{align*}
			&f_1=-(\eta_{\epsilon}\ast v_{2\epsilon})\cdot\nabla v_{2\epsilon}- v_{2\epsilon}\cdot\nabla v_1- v_1\cdot\nabla v_{2\epsilon}+(\eta_{\epsilon}\ast H_{2\epsilon})\cdot\nabla H_{2\epsilon}\\
			&~~~~~~~+ H_{2\epsilon}\cdot\nabla H_1+ H_1 \cdot\nabla H_{2\epsilon}-v_1\cdot\nabla v_1+ H_1\cdot\nabla H_1,\\
			&f_2=-(\eta_{\epsilon}\ast v_{2\epsilon})\cdot\nabla H_{2\epsilon}-v_{2\epsilon}\cdot\nabla H_1- v_1\cdot\nabla H_{2\epsilon}+(\eta_{\epsilon}\ast H_{2\epsilon})\cdot\nabla v_{2\epsilon}\\
			&~~~~~~~+ H_{2\epsilon}\cdot\nabla v_1+H_1\cdot\nabla v_{2\epsilon}- v_1\cdot\nabla H_1+H_1\cdot\nabla v_1.
\end{align*}
This, along with the maximal regularity of Stokes operator, yields
\begin{align*}
(\partial_t v_{2\epsilon},\partial_t H_{2\epsilon})\in L^2(0,T;{H}^{-1}(\mathbb R^3)).
\end{align*}
Thus, by the Lemma \ref{3-D ns equations}, one has $(v_{2\epsilon},H_{2\epsilon})\in C([0,T];L^2(\mathbb R^3))$.
\medskip

\textbf{Step 2:}~ {\em \crm{The existence time interval of the} energy weak solution pair $(v_{2\epsilon},H_{2\epsilon})$ constructed in Step 1 can be extended to $(0,\infty)$.}  Assume by contradiction that the survival time interval  of $(v_{2\epsilon},H_{2\epsilon})$  is $(0,T^*)$ with some finite $T^*$. First of all,  notice that by repeating the proof process of Lemma \ref{Energy Estimate}, one has
\begin{equation*}\label{energy estimate}
		\|(v_{2\epsilon},H_{2\epsilon})\|_{L^\infty(0,T^*;L^2(\mathbb R^3)) }+\|(v_{2\epsilon},H_{2\epsilon})\|_{L^2(0,T^*;\dot H^1(\mathbb R^3))}\leq \crm{C{T^*}^{\frac{3}{4}},}
\end{equation*}
\crm{with some constant $C$ depending only on $\|v_0\|_{L^3(\mathbb R^3)},\|h_0\|_{L^3(\mathbb R^3)}$.}
\crm{Moreover}, Step 1 shows
$$
(v_{2\epsilon},H_{2\epsilon})\in C([0,T^*];L^2(\mathbb R^3)).
$$
Therefore, $\|(v_{2\epsilon},H_{2\epsilon})(\cdot,T^*)\|_{L^2(\mathbb R^3)}$ is well-defined and bounded. Then, we reason as in the proof of Step 1 to conclude that $(v_{2\epsilon},H_{2\epsilon})$ can be extended after time $T^*$, contradicting the fact that $T^*$ is a maximum.

Now, we claim
\begin{align}\label{-2/3}
\|\partial_t v_{2\epsilon},\partial_t H_{2\epsilon}\|_{ L^2(0,T;H^{-\frac{3}{2}}(\mathbb R^3))}\leq C(\|v_0\|_{L^3(\mathbb R^3)},\|h_0\|_{\mathbb R^3},T)
\end{align}
 for any $T\in(0,+\infty)$.
To derive the desired result,  we need to verify that $(\Delta v_{2\epsilon},\Delta H_{2\epsilon})$ and each term of $f_1,f_2$ are uniformly bounded in $L^2(0,T;{H}^{-\frac{3}{2}}(\mathbb R^3))$, respectively.
For $\Delta v_{2\epsilon}$, employing the H\"{o}lder's inequality, we can get
\begin{align*}
\|\Delta v_{2\epsilon}\|^2_{L^2(0,T;H^{-1}(\mathbb R^3))}&=\int_{0}^{T}\|\Delta v_{2\epsilon}\|^2_{H^{-1}(\mathbb R^3)}\,{\rm d}t=\int_{0}^{T}\Big(\sup_{\stackrel{\varphi\in H^{1}(\mathbb R^3),}{\varphi\neq 0}}\frac{|\langle\Delta v_{2\epsilon},\varphi\rangle|}{\|\varphi\|_{H^{1}(\mathbb R^3)}}\Big)^2\,{\rm d}t\\
&\leq \int_{0}^{T}\|\nabla v_{2\epsilon}\|^2_{L^2(\mathbb R^3)}\,{\rm d}t\leq \|v_{2\epsilon}\|^2_{L^2(0,T;\dot {H}^1(\mathbb R^3))}.
\end{align*}
For $(v_{2\epsilon})^\epsilon\nabla v_{2\epsilon}$, by the H\"{o}lder's inequality, one has
\begin{align*}
\|(v_{2\epsilon})^\epsilon\nabla v_{2\epsilon}\|_{H^{-\frac{3}{2}}(\mathbb R^3)}&=\sup_{\stackrel{\varphi\in H^{\frac{3}{2}}(\mathbb R^3),}{\varphi\neq 0}}\frac{|\langle v_{2\epsilon}\nabla v_{2\epsilon},\varphi\rangle|}{\|\varphi\|_{H^{\frac{3}{2}}(\mathbb R^3)}}\leq\frac{|\langle v_{2\epsilon} v_{2\epsilon},\nabla\varphi\rangle|}{\|\nabla \varphi\|_{H^{\frac{1}{2}}(\mathbb R^3)}}\\
&\leq\frac{\|v_{2\epsilon}\|_{L^2(\mathbb R^3)}\|v_{2\epsilon}\|_{L^6(\mathbb R^3)}\|\nabla\varphi\|_{L^3(\mathbb R^3)}}{\|\nabla \varphi\|_{H^{\frac{1}{2}}(\mathbb R^3)}}\\
&\leq \|v_{2\epsilon}\|_{L^2(\mathbb R^3)}\|\nabla v_{2\epsilon}\|_{L^2(\mathbb R^3)},
\end{align*}
thus
\begin{align*}
\|v_{2\epsilon}\nabla v_{2\epsilon}\|_{L^2(0,T;H^{-\frac{3}{2}}(\mathbb R^3))}&=\Big(\int_{0}^{T}\|v_{2\epsilon}\nabla v_{2\epsilon}\|^2_{H^{-\frac{3}{2}}(\mathbb R^3)}\,{\rm d}t\Big)^{\frac{1}{2}}\\
&\leq\Big(\int_{0}^{T}\|v_{2\epsilon}\|^2_{L^2(\mathbb R^3)}\|\nabla v_{2\epsilon}\|^2_{L^2(\mathbb R^3)}\,{\rm d}t\Big)^{\frac{1}{2}}\\
&\leq\|v_{2\epsilon}\|_{L^\infty (0,T;L^2(\mathbb R^3))}\| v_{2\epsilon}\|_{L^2(0,T;H^1(\mathbb R^3))}.
\end{align*}
For $v_{2\epsilon}\nabla v_1$, by the H\"{o}lder's inequality, one has
\begin{align*}
\|v_{2\epsilon}\nabla v_1\|_{H^{-1}(\mathbb R^3)}&=\sup_{\stackrel{\varphi\in H^{1}(\mathbb R^3),}{\varphi\neq 0}}\frac{|\langle v_{2\epsilon}\nabla v_1,\varphi\rangle|}{\|\varphi\|_{H^{1}(\mathbb R^3)}}\leq\frac{|\langle v_{2\epsilon} v_1,\nabla\varphi\rangle|}{\|\nabla \varphi\|_{L^2(\mathbb R^3)}}\\
&\leq\frac{\|v_{2\epsilon}\|_{L^6(\mathbb R^3)}\|v_1\|_{L^3(\mathbb R^3)}\|\nabla\varphi\|_{L^2(\mathbb R^3)}}{\|\nabla \varphi\|_{L^2(\mathbb R^3)}}\\
&\leq \|\nabla v_{2\epsilon}\|_{L^2(\mathbb R^3)}\|v_1\|_{L^3(\mathbb R^3)},
\end{align*}
thus
\begin{align*}
\|v_{2\epsilon}\nabla v_1\|_{L^2(0,T;H^{-1}(\mathbb R^3))}&=\Big(\int_{0}^{T}\|v_{2\epsilon}\nabla v_1\|^2_{H^{-1}(\mathbb R^3)}\,{\rm d}t\Big)^{\frac{1}{2}}\\
&\leq\Big(\int_{0}^{T}\|\nabla v_{2\epsilon}\|^2_{L^2(\mathbb R^3)}\| v_1\|^2_{L^3(\mathbb R^3)}\,{\rm d}t\Big)^{\frac{1}{2}}\\
&\leq\|v_{2\epsilon}\|_{L^2(0,T;\dot{H}^1(\mathbb R^3))}\| v_1\|_{L^\infty(0,T;L^3(\mathbb R^3))}.\\
\end{align*}
For $v_1\nabla v_{2\epsilon}$, one has
\begin{align*}
\|v_1\nabla v_{2\epsilon}\|_{H^{-1}(\mathbb R^3)}&=\sup_{\stackrel{\varphi\in H^{1}(\mathbb R^3),}{\varphi\neq 0}}\frac{|\langle v_1\nabla v_{2\epsilon},\varphi\rangle|}{\|\varphi\|_{H^{1}(\mathbb R^3)}}\\
&\leq\frac{\|v_1\|_{L^3(\mathbb R^3)}\|\nabla v_{2\epsilon}\|_{L^2(\mathbb R^3)}\|\varphi\|_{L^6(\mathbb R^6)}}{\|\varphi\|_{H^{1}(\mathbb R^3)}}\\
&\leq\|v_1\|_{L^3(\mathbb R^3)}\|\nabla v_{2\epsilon}\|_{L^2(\mathbb R^3)}.
\end{align*}
From which we have
\begin{align*}
\|v_1\nabla v_{2\epsilon}\|_{L^2(0,T;H^{-1}(\mathbb R^3))}&=\Big(\int_{0}^{T}\|v_1\nabla v_{2\epsilon}\|^2_{H^{-1}(\mathbb R^3)}\,{\rm d}t\Big)^{\frac{1}{2}}\\
&\leq\Big(\int_{0}^{T}\|v_1\|^2_{L^3(\mathbb R^3)}\|\nabla v_{2\epsilon}\|^2_{L^2(\mathbb R^3)}\,{\rm d}t\Big)^{\frac{1}{2}}\\
&\leq\|v_1\|_{L^\infty(0,T;L^3(\mathbb R^3))}\|v_{2\epsilon}\|_{L^2(0,T;\dot {H}^1(\mathbb R^3))}.
\end{align*}
For $v_1\nabla v_1$,
\begin{align*}
\|v_1\nabla v_1\|_{H^{-\frac{3}{2}}(\mathbb R^3)}&=\sup_{\stackrel{\varphi\in H^{\frac{3}{2}}(\mathbb R^3),}{\varphi\neq 0}}\frac{|\langle v_1\nabla v_1,\varphi\rangle|}{\|\varphi\|_{H^{\frac{3}{2}}(\mathbb R^3)}}
\leq\frac{|\langle v_1 v_1,\nabla\varphi\rangle|}{\|\nabla \varphi\|_{H^{\frac{1}{2}}(\mathbb R^3)}}\\
&\leq\frac{\|v_1\|_{L^3(\mathbb R^3)}\|v_1\|_{L^3(\mathbb R^3)}\|\nabla\varphi\|_{L^3(\mathbb R^3)}}{\|\nabla \varphi\|_{H^{\frac{1}{2}}(\mathbb R^3)}}\\
&\leq \|v_1\|^2_{L^3(\mathbb R^3)},
\end{align*}
it follows that
\begin{align*}
\|v_1\nabla v_1\|_{L^2(0,T;H^{-\frac{3}{2}}(\mathbb R^3))}&=\Big(\int_{0}^{T}\|v_1\nabla v_1\|^2_{H^{-\frac{3}{2}}(\mathbb R^3)}\,{\rm d}t\Big)^{\frac{1}{2}}
\leq\left(\int_{0}^{T}\|v_1\|^4_{L^3(\mathbb R^3)}\,{\rm d}t\right)^{\frac{1}{2}}\\
&\leq\|v_1\|_{L^4(0,T;L^3(\mathbb R^3))}\|v_1\|_{L^\infty(0,T;L^3(\mathbb R^3))}\leq T^{\frac{1}{4}}\|v_1\|^2_{L^\infty(0,T;L^3(\mathbb R^3))}.
\end{align*}
Using a similar process, one can demonstrate that the other terms of $f_1,f_2$ are uniformly bounded in $L^2(0,T;H^{-\frac{3}{2}}(\mathbb R^3))$, which, along with the above estimates, conclude the claim \eqref{-2/3}.
\end{proof}
\medskip

\subsection{The proof of Theorem \ref{thm 1.1}}

In the previous subsection, we constructed the global weak solution sequence to the approximate equation \eqref{smooth equation}, which satisfies the energy estimates \eqref{1-4-1}. In this section, we will utilize energy estimates for the approximate solutions combined with Lemma \ref{RTtheorem_23} to derive a \crm{global} weak solution of system \eqref{v_2}.

\begin{proof}[Proof of Theorem \ref{thm 1.1}]
To begin with,  the sequence $(v_{2\epsilon},H_{2\epsilon})$ is bound in $X_T$ uniformly in $\epsilon$. Thus, we can find a subsequence of $(v_{2\epsilon},H_{2\epsilon})$, still denoted by $(v_{2\epsilon},H_{2\epsilon})$ such that as $\epsilon\rightarrow 0$,
\begin{equation}\label{strong and weak convergence}
	\begin{aligned}
&(v_{2\epsilon},H_{2\epsilon})\stackrel{*}{\rightharpoonup}(v_2,H_2)\quad{\rm in}\quad L^\infty(0,T;L^2(\mathbb R^3)),\\
&(v_{2\epsilon},H_{2\epsilon}){\rightharpoonup}(v_2,H_2)\quad\ \ {\rm in}\quad L^2(0,T;\dot{H}^1(\mathbb R^3)).
\end{aligned}
\end{equation}
 The Lemma \ref{RTtheorem_23}, along with \eqref{1-4-1}, \eqref{-2/3},  implies there is subsequence $(v_{2\epsilon},H_{2\epsilon})$,  still denoted by $(v_{2\epsilon},H_{2\epsilon})$ , such that
\begin{equation}\label{L2 strong convergence}
(v_{2\epsilon},H_{2\epsilon})\rightarrow (v_2,H_2)\quad {\rm in}\quad \crm{L^2(0,T; L^2_{\rm loc}(\R^3)),}
\end{equation}
\crm{for any $T<\infty$.} To conclude the proof, we need to prove that $(v,H)=(v_1+v_2, H_1+H_2)$ is a global weak solution of  the Cauchy problem \eqref{EQ-tlh}-\eqref{initial value} in the sense of Definition \ref{def1}. The strategy proceeds in three steps: (i) prove that $(v_2,H_2)$ satisfies system \eqref{v_2} in the sense of distributions, (ii) prove that $(v_2,H_2)$ satisfies the local energy inequality \eqref{local energy estimate}, (iii) prove that the global energy inequality \eqref{global energy inequality} holds.

\textbf{Step (i):}~ {\em Prove that $(v_2,H_2)$ satisfies system \eqref{v_2} in the sense of distributions.}
To achieve this, we multiply both sides of \eqref{smooth equation} by $w\in C^\infty_{0,\sigma}(\mathbb R^3\times (0,\infty))$ and integrate by parts to obtain
\begin{subequations}
\begin{equation}\label{subequation1}
	\begin{aligned}
		&\int_{0}^{\infty}\int_{\mathbb R^3}v_{2\epsilon}\partial_t w \mathrm{d}x \mathrm{d}t-\int_{0}^{\infty}\nabla v_{2\epsilon}\nabla w \mathrm{d}x\mathrm{d}t=\int_{0}^{\infty}\int_{\mathbb R^3}\big((\eta_{\epsilon}\ast v_{2\epsilon})\nabla w v_{2\epsilon}+v_{2\epsilon}\nabla w v_1+ v_1\nabla w v_{2_\epsilon}\\
		&~~~~~-(\eta_{\epsilon}\ast H_{2\epsilon})\nabla w H_{2\epsilon}- H_{2\epsilon}\nabla w H_1- H_1\nabla w H_{2\epsilon}+ v_1\nabla w v_1- H_1\nabla w H_1\big)\mathrm{d}x\mathrm{d}t,
\end{aligned}
\end{equation}
\begin{equation}\label{subequation2}
\begin{aligned}
		&\int_{0}^{\infty}\int_{\mathbb R^3}H_{2\epsilon}\partial_t w \mathrm{d}x \mathrm{d}t-\int_{0}^{\infty}\nabla H_{2\epsilon}\nabla w \mathrm{d}x\mathrm{d}t=\int_{0}^{\infty}\int_{\mathbb R^3}\big((\eta_{\epsilon}\ast v_2)\nabla w H_{2\epsilon}+v_{2\epsilon}\nabla w H_1+ v_1\nabla w H_{2_\epsilon}\\
		&~~~~~-(\eta_{\epsilon}\ast H_2)\nabla w v_{2\epsilon}- H_{2\epsilon}\nabla w v_1- H_1\nabla w v_{2\epsilon}+ v_1\nabla w H_1- H_1\nabla w v_1\big)\mathrm{d}x\mathrm{d}t.
\end{aligned}
\end{equation}
\end{subequations}
From \eqref{strong and weak convergence}, it is known that
$$\lim_{\epsilon\to 0}\int_{0}^{\infty}\int_{\mathbb R^3}v_{2\epsilon}\partial_t w\mathrm{d}x\mathrm{d}t =\int_{0}^{\infty}\int_{\mathbb R^3}v_2\partial_t w\mathrm{d}x\mathrm{d}t,\quad \lim_{\epsilon\to 0}\int_{0}^{\infty}\int_{\mathbb R^3}H_{2\epsilon}\partial_t w\mathrm{d}x\mathrm{d}t =\int_{0}^{\infty}\int_{\mathbb R^3}H_2\partial_t w\mathrm{d}x\mathrm{d}t,$$
$$\lim_{\epsilon\to 0}\int_{0}^{\infty}\int_{\mathbb R^3}\nabla H_{2\epsilon}\nabla w \mathrm{d}x\mathrm{d}t=\int_{0}^{\infty}\int_{\mathbb R^3}\nabla H_2\nabla w \mathrm{d}x\mathrm{d}t,\quad\lim_{\epsilon\to 0}\int_{0}^{\infty}\int_{\mathbb R^3}\nabla v_{2\epsilon}\nabla w \mathrm{d}x\mathrm{d}t=\int_{0}^{\infty}\int_{\mathbb R^3}\nabla v_2\nabla w \mathrm{d}x\mathrm{d}t.$$
For the term $\int_{0}^{\infty}\int_{\mathbb R^3}(\eta_{\epsilon}\ast v_2)\nabla w H_{2\epsilon}\mathrm{d}x\mathrm{d}t$ of \eqref{subequation1}, one has
\begin{equation*}
	\begin{aligned}
		&\Big|\int_{0}^{\infty}\int_{\mathbb R^3}(\eta_{\epsilon}\ast v_{2\epsilon})\nabla w v_{2\epsilon}\mathrm{d}x\mathrm{d}t-\int_{0}^{\infty}\int_{\mathbb R^3} v_2\nabla w v_2\mathrm{d}x\mathrm{d}t\Big|\\
		&\leq \int_{0}^{\infty}\int_{\mathbb R^3}\big|(\eta_{\epsilon}\ast v_{2\epsilon})\nabla w v_{2\epsilon}-(\eta_{\epsilon}\ast v_{2\epsilon})\nabla w v_2\big| \mathrm{d}x\mathrm{d}t\\
&\quad+\int_{0}^{\infty}\int_{\mathbb R^3}\big|(\eta_{\epsilon}\ast v_{2\epsilon})\nabla w v_2\mathrm{d}x\mathrm{d}t-v_2\nabla w v_2\big|\mathrm{d}x\mathrm{d}t\\
&\leq\int_{0}^{\infty}\|\eta_{\epsilon}\ast v_{2\epsilon}\|_{L^2(\mathbb R^3)}\|\nabla w\|_{L^\infty(\mathbb R^3)}\|v_{2\epsilon}-v_2\|_{L^2_{\rm loc}(\mathbb R^3)}\mathrm{d}t\\
&\quad+\int_{0}^{\infty}\|\eta_{\epsilon}\ast v_{2\epsilon}-v_2\|_{L^2_{\rm loc}(\mathbb R^3)}\|\nabla w\|_{L^\infty(\mathbb R^3)}\|v_2\|_{L^2(\mathbb R^3)}\mathrm{d}t\\
&\leq \|\eta_{\epsilon}\ast v_{2\epsilon}\|_{L^\infty (0,\infty;L^2(\mathbb R^3))}\|\nabla w\|_{L^2(0,\infty;L^\infty(\mathbb R^3))}\|v_{2\epsilon}-v_2\|_{\crm{L^2(0,T;L^2_{\rm loc}(\mathbb R^3))}}\\
&\quad+\|\eta_{\epsilon}\ast v_{2\epsilon}-v_2\|_{L^2(0,\infty;L^2(\mathbb R^3))}\|\nabla w\|_{L^2(0,\infty;L^\infty(\mathbb R^3))}\|v_2\|_{L^\infty(0,\infty;L^2(\mathbb R^3))}\rightarrow 0,
	\end{aligned}
\end{equation*}
thus
$$\lim_{\epsilon\to 0}\int_{0}^{\infty}\int_{\mathbb R^3}(\eta_{\epsilon}\ast v_{2\epsilon})\nabla w v_{2\epsilon}\mathrm{d}x\mathrm{d}t=\int_{0}^{\infty}\int_{\mathbb R^3} v_2\nabla w v_2\mathrm{d}x\mathrm{d}t.$$
For $\int_{0}^{\infty}\int_{\mathbb R^3} v_{2\epsilon}\nabla w v_1\mathrm{d}x\mathrm{d}t$, one has
\begin{equation*}
	\begin{aligned}
&\Big|\int_{0}^{\infty}\int_{\mathbb R^3} v_{2\epsilon}\nabla w v_1\mathrm{d}x\mathrm{d}t-\int_{0}^{\infty}\int_{\mathbb R^3} v_2\nabla w v_1\mathrm{d}x\mathrm{d}t\Big|\leq \Big|\int_{0}^{\infty}\|v_{2\epsilon}-v_2\|_{L^2_{\rm loc}(\mathbb R^3)}\|v_1\|_{L^3(\mathbb R^3)}\|\nabla w\|_{L^6(\mathbb R^3)}\mathrm{d}t\Big|\\
		&\leq \| v_{2\epsilon}-v_2\|_{\crm{L^2(0,T;L^2_{\rm loc}(\mathbb R^3))}}\|v_1\|_{L^\infty(0,\infty; L^3(\mathbb R^3))}\|\nabla w\|_{L^2(0,\infty;L^6(\mathbb R^3))}\rightarrow 0,
	\end{aligned}
\end{equation*}
therefore
$$\lim_{\epsilon\to 0}\int_{0}^{\infty}\int_{\mathbb R^3} v_{2\epsilon}\nabla w v_1\mathrm{d}x\mathrm{d}t=\int_{0}^{\infty}\int_{\mathbb R^3} v_2\nabla w v_1\mathrm{d}x\mathrm{d}t.$$
 Similarly, we have
 \begin{equation*}
 \begin{aligned}
 &\lim_{\epsilon\to 0}\int_{0}^{\infty}\int_{\mathbb R^3}\big(v_1\nabla w v_{2_\epsilon}-(\eta_{\epsilon}\ast H_{2\epsilon})\nabla w H_{2\epsilon}- H_{2\epsilon}\nabla w H_1- H_1\nabla w H_{2\epsilon}\big)\mathrm{d}x\mathrm{d}t\\
 &=\int_{0}^{\infty}\int_{\mathbb R^3}\big(v_1\nabla w v_{2}-H_{2}\nabla w H_{2}- H_{2}\nabla w H_1- H_1\nabla w H_{2}\big)\mathrm{d}x\mathrm{d}t,
 \end{aligned}
 \end{equation*}
 \begin{equation*}
 \begin{aligned}
&\lim_{\epsilon\to 0}\int_{0}^{\infty}\int_{\mathbb R^3}\big((\eta_{\epsilon}\ast v_2)\nabla w H_{2\epsilon}+v_{2\epsilon}\nabla w H_1+ v_1\nabla w H_{2_\epsilon}-(\eta_{\epsilon}\ast H_2)\nabla w v_{2\epsilon}\\
&~~~~~~~~~~~~~~~~~- H_{2\epsilon}\nabla w v_1- H_1\nabla w v_{2\epsilon}\big)\mathrm{d}x\mathrm{d}t\\
&=\int_{0}^{\infty}\int_{\mathbb R^3}\big( v_2\nabla w H_2+v_2\nabla w H_1+ v_1\nabla w H_2-H_2\nabla w v_2- H_2\nabla w v_1- H_1\nabla w v_2\big)\mathrm{d}x\mathrm{d}t.
 \end{aligned}
 \end{equation*}
 In summary, we can conclude that
\begin{equation*}
	\begin{aligned}
		&\int_{0}^{\infty}\int_{\mathbb R^3}v_{2}\partial_t w \mathrm{d}x \mathrm{d}t-\int_{0}^{\infty}\nabla v_{2}\nabla w \mathrm{d}x\mathrm{d}t=\int_{0}^{\infty}\int_{\mathbb R^3}\big( v_{2}\nabla w v_{2}+v_{2}\nabla w v_1+ v_1\nabla w v_{2}\\
		&~~~~~-H_{2}\nabla w H_{2}- H_{2}\nabla w H_1- H_1\nabla w H_{2}+ v_1\nabla w v_1- H_1\nabla w H_1\big)\mathrm{d}x\mathrm{d}t,\\
		&\int_{0}^{\infty}\int_{\mathbb R^3}H_{2}\partial_t w \mathrm{d}x \mathrm{d}t-\int_{0}^{\infty}\nabla H_{2}\nabla w \mathrm{d}x\mathrm{d}t=\int_{0}^{\infty}\int_{\mathbb R^3}\big( v_2\nabla w H_{2}+v_{2}\nabla w H_1+ v_1\nabla w H_{2}\\
		&~~~~~- H_2\nabla w v_{2}- H_{2}\nabla w v_1- H_1\nabla w v_{2}+ v_1\nabla w H_1- H_1\nabla w v_1\big)\mathrm{d}x\mathrm{d}t.
\end{aligned}
\end{equation*}
Thus, $(v_2,H_2)$ satisfies \eqref{v_2} in the sense of distributions.

\textbf{Step (ii):}~ {\em Prove that $(v_2,H_2)$ satisfies the local energy inequality \eqref{local energy estimate}.}
To do this, we take the inner product of \eqref{smooth equation} with $(v_{2\epsilon}\phi,H_{2\epsilon}\phi)$ and integrate over space-time and find that
\begin{align}\label{approximate energy inequality}
	\begin{split}
		&\int_{\mathbb R^3}(|v_{2\epsilon}|^2+|H_{2\epsilon}|^2)\phi\,{\rm d}x+2\int_{0}^{\infty}\int_{\mathbb R^3}(|\nabla v_{2\epsilon}|^2+|\nabla H_{2\epsilon}|^2)\phi\,{\rm d}x{\rm d}s\\
		&=\int_{0}^{\infty}\int_{\mathbb R^3}(|v_{2\epsilon}|^2+|H_{2\epsilon}|^2)(\partial_t\phi+\Delta\phi)\,{\rm d}x{\rm d}s+\int_{0}^{\infty}\int_{\mathbb R^3}(v_{2\epsilon})^\epsilon\nabla\phi(|v_{2\epsilon}|^2+|H_{2\epsilon}|^2)\,{\rm d}x{\rm d}s\\
		&\quad-2\int_{0}^{\infty}\int_{\mathbb R^3}v_{2\epsilon}H_{2\epsilon}(H_{2\epsilon})^\epsilon\nabla\phi \,{\rm d}x{\rm d}s+\int_{0}^{\infty}\int_{\mathbb R^3}\Pi^\epsilon v_{2\epsilon}\nabla\phi \,{\rm d}x{\rm d}s\\
		&\quad+\int_{0}^{\infty}\int_{\mathbb R^3}\Big(v_1\otimes v_{2\epsilon}+v_{2\epsilon}\otimes v_1+v_1\otimes v_1-H_1\otimes H_{2\epsilon}-H_{2\epsilon}\otimes H_1-H_1\otimes H_1\Big):\\
		&~~~~~~~~~~~~~(\nabla v_{2\epsilon}\phi+v_{2\epsilon}\otimes\nabla\phi)\,{\rm d}x{\rm d}s\\
		&\quad+\int_{0}^{\infty}\int_{\mathbb R^3}\Big(H_1\otimes v_{2\epsilon}+H_{2\epsilon}\otimes v_1+H_1\otimes v_1-v_1\otimes H_{2\epsilon}-v_{2\epsilon}\otimes H_1-v_1\otimes H_1\Big):\\
		&~~~~~~~~~~~~~(\nabla H_{2\epsilon}\phi+H_{2\epsilon}\otimes\nabla\phi)\,{\rm d}x{\rm d}s.
	\end{split}
\end{align}
 Here, $\phi\in C^\infty_0(\mathbb R^3\times (0,\infty))$. In the following, we, for short, denote $\eta_{\epsilon}\ast v_{2\epsilon}$ by $(v_{2\epsilon})^\epsilon$ and $\eta_{\epsilon}\ast H_{2\epsilon}$ by $(H_{2\epsilon})^\epsilon$ respectively. Note that $(v_{2\epsilon},H_{2\epsilon})\rightarrow (v_2,H_2)$ in $L^2(0,T;L^2_{\rm loc}(\mathbb R^3))$ and $(v_{2\epsilon},H_{2\epsilon})$ is uniformly bounded in $\epsilon$ in $\crm{L^{\frac{10}{3}}(\mathbb R^3\times(0,T))}$. This, along with  the interpolation inequality, yields,
\begin{equation*}
	\|(v_{2\epsilon},H_{2\epsilon})-(v_2,H_2)\|_{L^3(\mathbb R^3)}\leq \|(v_{2\epsilon},H_{2\epsilon})-(v_2,H_2)\|^{\frac{1}{6}}_{L^2(\mathbb R^3)}\|(v_{2\epsilon},H_{2\epsilon})-(v_2,H_2)\|^{\frac{5}{6}}_{L^\frac{10}{3}(\mathbb R^3)}
\end{equation*}
thus, as $\epsilon\rightarrow 0$,
\begin{equation}\label{L3convergence}
\begin{aligned}
(v_{2\epsilon},H_{2\epsilon})\rightarrow (v_2,H_2)\quad{\rm in}\quad L^3(0,T;L^3_{\rm loc}(\mathbb R^3)).
\end{aligned}
\end{equation}
On the other hand,
\begin{equation*}
	\begin{aligned}
		&\|\big((v_{2\epsilon})^\epsilon,(H_{2\epsilon})^\epsilon\big)- (v_2,H_2)\|_{L^3(\mathbb R^3)}\\
&\leq\|\big(\eta_{\epsilon}\ast(v_{2\epsilon}-v_2),\eta_{\epsilon}\ast(H_{2\epsilon}-H_2)\big)\|_{L^3(\mathbb R^3)}+\|\big(\eta_{\epsilon}\ast v_2,\eta_{\epsilon}\ast H_2\big)-(v_2,H_2)\|_{L^3(\mathbb R^3)}\\
&\leq C\|(v_{2\epsilon}-v_2,H_{2\epsilon}-H_2)\|_{L^3(\mathbb R^3)}+\|\big(\eta_{\epsilon}\ast v_2,\eta_{\epsilon}\ast H_2\big)-(v_2,H_2)\|_{L^3(\mathbb R^3)}\rightarrow 0,
	\end{aligned}
\end{equation*}
it follows that
$$
\big((v_{2\epsilon})^\epsilon,(H_{2\epsilon})^\epsilon\big)\rightarrow (v_2,H_2)\quad{\rm in}\quad L^3(0,T;L^3_{\rm loc}(\mathbb R^3)).
$$
Next we show that all the terms in \eqref{approximate energy inequality} converge to the corresponding limiting terms in \eqref{local energy estimate}. In fact, for the first term in the left-hand side of \eqref{approximate energy inequality},from \eqref{L2 strong convergence}, one has
\begin{equation*}
	\begin{aligned}
		&\Big|\int_{\mathbb R^3}(|v_{2\epsilon}|^2+|H_{2\epsilon}|^2)\phi\,{\rm d}x-\int_{\mathbb R^3}(|v_2|^2+|H_2|^2)\phi\,{\rm d}x\Big|\\
		&\leq\|\phi\|_{L^\infty(\mathbb R^3)}\| v_{2\epsilon}-v_2\|^2_{L^2(\mathbb R^3)}+\|\phi\|_{L^\infty(\mathbb R^3)}\|H_{2\epsilon}-H_2\|^2_{L^2(\mathbb R^3)}\rightarrow 0.
	\end{aligned}
\end{equation*}
Similarly, for the first term in the right-hand side of \eqref{approximate energy inequality}, one has as $\epsilon\rightarrow0$,
\begin{align*}
\Big|\int_{0}^{\infty}\int_{\mathbb R^3}(|v_{2\epsilon}|^2+|H_{2\epsilon}|^2)(\partial_t\phi+\Delta\phi)\,{\rm d}x{\rm d}s-\int_{0}^{\infty}\int_{\mathbb R^3}(|v_{2}|^2+|H_{2}|^2)(\partial_t\phi+\Delta\phi)\,{\rm d}x{\rm d}s\Big|\rightarrow 0.
\end{align*}
Since $(\nabla v_{2\epsilon},\nabla H_{2\epsilon} )\rightharpoonup(v_2,H_2)$ in $\crm{L^2(\R^3\times(0,T))}$, by Fatou Lemma, we have
\begin{align*}
\int_{0}^{\infty}\int_{\mathbb R^3}(|\nabla v_{2}|^2+|\nabla H_{2}|^2)\phi\,{\rm d}x{\rm d}s\leq\liminf_{\epsilon\rightarrow 0}\int_{0}^{\infty}\int_{\mathbb R^3}(|\nabla v_{2\epsilon}|^2+|\nabla H_{2\epsilon}|^2)\phi\,{\rm d}x{\rm d}s.
\end{align*}
We now consider the second term on the right-hand of \eqref{approximate energy inequality}. By the properties of mollifiers, as $\epsilon\rightarrow0$ and  H\"{o}lder's inequality, one has
\begin{equation*}
	\begin{aligned}
		&\Big|\int_{0}^{\infty}\int_{\mathbb R^3}\Big((v_{2\epsilon})^\epsilon\nabla\phi|v_{2\epsilon}|^2-v_2\nabla \phi|v_2|^2\Big)\,{\rm d}x{\rm d}s\Big|\\
		&\leq \|\nabla\phi\|_{L^\infty(\mathbb R^3)}\int_{0}^{\infty}\int_{\mathbb R^3}\big|(v_{2\epsilon})^\epsilon|v_{2\epsilon}|^2-v_2|v_2|^2\big|\,{\rm d}x{\rm d}s\\
		&\leq C\int_{0}^{\infty}\int_{\mathbb R^3}|(v_{2\epsilon})^\epsilon-v_2|^2(v_{2\epsilon})^\epsilon+2v_2|(v_{2\epsilon})^\epsilon-v_2|(v_{2\epsilon})^\epsilon+|v_2|^2|(v_{2\epsilon})^\epsilon-v_2|\,{\rm d}x{\rm d}s\\
		&\leq C\Big(\|(v_{2\epsilon})^\epsilon-v_2\|^2_{\crm{L^3(0,T;L^3_{\rm loc}(\mathbb R^3))}}\|(v_{2\epsilon})^\epsilon\|_{\crm{L^3(0,T;L^3_{\rm loc}(\mathbb R^3))}}
+\|v_2\|^2_{\crm{L^3(0,T;L^3_{\rm loc}(\mathbb R^3))}}\|(v_{2\epsilon})^\epsilon-v_2\|_{\crm{L^3(0,T;L^3_{\rm loc}(\mathbb R^3))}}\\
&~~+2\|v_2\|_{\crm{L^3(0,T;L^3_{\rm loc}(\mathbb R^3))}}\|(v_{2\epsilon})^\epsilon-v_2\|_{\crm{L^3(0,T;L^3_{\rm loc}(\mathbb R^3))}}\|(v_{2\epsilon})^\epsilon\|_{\crm{L^3(0,T;L^3_{\rm loc}(\mathbb R^3))}}\Big)\rightarrow 0.
	\end{aligned}
\end{equation*}
Similarly,
$$
\int_{0}^{\infty}\int_{\mathbb R^3}(v_{2\epsilon})^\epsilon\nabla\phi|H_{2\epsilon}|^2\,{\rm d}x{\rm d}s\rightarrow\int_{0}^{\infty}\int_{\mathbb R^3}v_{2}\nabla\phi|H_{2}|^2\,{\rm d}x{\rm d}s.
$$
For the third term on the right-hand of \eqref{approximate energy inequality}, through the same process, we have
\begin{equation*}
	\begin{aligned}
\int_{0}^{\infty}\int_{\mathbb R^3}(v_{2\epsilon}H_{2\epsilon})(H_{2\epsilon})^\epsilon\nabla\phi\,{\rm d}x{\rm d}s\rightarrow\int_{0}^{\infty}\int_{\mathbb R^3}(v_2H_2)H_2\nabla\phi \,{\rm d}x{\rm d}s.
	\end{aligned}
\end{equation*}
In the following, we focus on the last two terms of \eqref{approximate energy inequality}. Observing that
\begin{equation*}
\begin{aligned}
&\Big|\int_{0}^{\infty}\int_{\mathbb R^3}v_{2\epsilon}\nabla v_1v_{2\epsilon}\phi \,{\rm d}x{\rm d}s-\int_{0}^{\infty}\int_{\mathbb R^3}v_2\nabla v_1v_2\phi \,{\rm d}x{\rm d}s\Big|\\
&\leq\int_{0}^{\infty}\int_{\mathbb R^3}\big|(v_{2\epsilon}-v_2)\nabla v_1v_{2\epsilon}\phi \big|\,{\rm d}x{\rm d}s+\int_{0}^{\infty}\int_{\mathbb R^3}\big|v_2\nabla v_1(v_{2\epsilon}-v_2)\phi\,{\rm d}x{\rm d}s\big|\\
&\leq\|v_{2\epsilon}-v_2\|_{\crm{L^3(0,T;L^3_{\rm loc}(\mathbb R^3))}}\|v_{2\epsilon}\|_{\crm{L^3(0,T;L^3_{\rm loc}(\mathbb R^3))}}\|\nabla v_1\phi\|_{\crm{L^3(0,T;L^3_{\rm loc}(\mathbb R^3))}}\\
&+\|v_{2\epsilon}-v_2\|_{\crm{L^3(0,T;L^3_{\rm loc}(\mathbb R^3))}}\|v_2\|_{\crm{L^3(0,T;L^3_{\rm loc}(\mathbb R^3))}}\|\nabla v_1\phi\|_{\crm{L^3(0,T;L^3_{\rm loc}(\mathbb R^3))}}\rightarrow 0.
\end{aligned}
\end{equation*}
Thus, we can obtain
\begin{equation*}
	\begin{aligned}
&\int_{0}^{\infty}\int_{\mathbb R^3}v_1\otimes v_{2\epsilon}:(\nabla v_{2\epsilon}\phi+v_{2\epsilon}\otimes\nabla\phi)\,{\rm d}x{\rm d}s=-\int_{0}^{\infty}\int_{\mathbb R^3}v_{2\epsilon}\nabla v_1v_{2\epsilon}\phi \,{\rm d}x{\rm d}s\\
&\rightarrow -\int_{0}^{\infty}\int_{\mathbb R^3}v_2\nabla v_1v_2\phi \,{\rm d}x{\rm d}s=\int_{0}^{\infty}\int_{\mathbb R^3}v_1\otimes v_2(\nabla v_2\phi+v_2\otimes\nabla\phi) \,{\rm d}x{\rm d}s.
	\end{aligned}
\end{equation*}
Similarly, we have
\begin{equation*}
	\begin{aligned}
		&\int_{0}^{\infty}\int_{\mathbb R^3}v_{2\epsilon}\otimes v_1:(\nabla v_{2\epsilon}\phi+v_{2\epsilon}\otimes\nabla\phi)\,{\rm d}x{\rm d}s=-\int_{0}^{\infty}\int_{\mathbb R^3}v_1\nabla v_{2\epsilon}v_{2\epsilon}\phi \,{\rm d}x{\rm d}s\\
&\rightarrow-\int_{0}^{\infty}\int_{\mathbb R^3}v_1\nabla v_2v_2\phi \,{\rm d}x{\rm d}s=\int_{0}^{\infty}\int_{\mathbb R^3}v_2\otimes v_1(\nabla v_2\phi+v_{2\epsilon}\otimes\nabla\phi)\,{\rm d}x{\rm d}s.
	\end{aligned}
\end{equation*}
Going through the same process, we assert that the last two terms of \eqref{approximate energy inequality} converge. Now we consider the pressure term, which is more complex. It follows from \eqref{smooth equation} that
\begin{align*}
 \Pi^\epsilon
 &= -\frac{1}{3}\big((v_{2\epsilon})^\epsilon v_{2\epsilon}+2v_{2\epsilon}v_1+|v_1|^2-(H_{2\epsilon})^\epsilon H_{2\epsilon}-2 H_{2\epsilon}H_1-|H_1|^2\big)\\
 &+\mathrm{p.v.}\int_{B(2)}K_{ij}(x-y)N_{ij}^{\epsilon}(y,s)\mathrm{d}y + \mathrm{p.v.}\int_{\R^3\setminus B(2)}[K_{ij}(x-y)-K_{ij}(-y)]N_{ij}^{\epsilon}(y,s)\mathrm{d}y,
\end{align*}
where
$$K(x)= \frac{1}{4\pi|x|},\quad K_{ij}= \partial_{ij}K,$$
\begin{align*}
N_{ij}^{\epsilon}(y,t)&=(v_{2\epsilon})_{i}^\epsilon (v_{2\epsilon})_j+(v_{2\epsilon})_i(v_1)_j+(v_1)_i(v_{2\epsilon})_j+(v_1)_i(v_1)_j\\
&-(H_{2\epsilon})_i^\epsilon (H_{2\epsilon})_j- (H_{2\epsilon})_i(H_1)_j-(H_1)_i(H_{2\epsilon})_j-(H_1)_i(H_1)_j.
\end{align*}
We need to prove that
$$\lim_{\epsilon\rightarrow 0}\|\Pi^\epsilon-\Pi\|_{\crm{L^{\frac{3}{2}}(0,T;L^{\frac{3}{2}}_{\rm loc}(\R^3))}}=0,$$
where
\begin{align*}
\Pi
 &=-\frac{1}{3}\big(|v_2|^2+2v_2v_1+|v_1|^2-|H_2|^2-2 H_2H_1-|H_1|^2\big)\\
 &+\mathrm{p.v.}\int_{B(2)}K_{ij}(x-y)N_{ij}(y,s)\mathrm{d}y + \mathrm{p.v.}\int_{\R^3\setminus B(2)}[K_{ij}(x-y)-K_{ij}(-y)]N_{ij}(y,s)\mathrm{d}y,
\end{align*}
\begin{align*}
 N_{ij}(y,t)&=(v_2)_i (v_2)_j+(v_2)_i(v_1)_j+(v_1)_i(v_2)_j+(v_1)_i(v_1)_j\\
&-(H_2)_i(H_2)_j- (H_2)_i(H_1)_j-(H_1)_i(H_2)_j-(H_1)_i(H_1)_j.
\end{align*}
To achieve this, \crm{we first notice that}
\begin{align*}
&\|N^\epsilon_{ij}-N_{ij}\|_{\crm{L^{\frac{3}{2}}(0,T;L^{\frac{3}{2}}_{\rm loc}(\R^3))}}\\
&\leq\|(v_{2\epsilon})^\epsilon_i(v_{2\epsilon}-v_2)_j\|_{\crm{L^{\frac{3}{2}}(0,T;L^{\frac{3}{2}}_{\rm loc}(\R^3))}}
+\|((v_{2\epsilon})^\epsilon-v_2)_i(v_2)_j\|_{\crm{L^{\frac{3}{2}}(0,T;L^{\frac{3}{2}}_{\rm loc}(\R^3))}}\\
&~~+\|(H_{2\epsilon})^\epsilon_i(H_{2\epsilon}-H_2)_j\|_{\crm{L^{\frac{3}{2}}(0,T;L^{\frac{3}{2}}_{\rm loc}(\R^3))}}+\|((H_{2\epsilon})^\epsilon-H_2)_i(H_2)_j\|_{\crm{L^{\frac{3}{2}}(0,T;L^{\frac{3}{2}}_{\rm loc}(\R^3))}}\\
&~~+\|(v_{2\epsilon}-v_2)_i(v_1)_j\|_{\crm{L^{\frac{3}{2}}(0,T;L^{\frac{3}{2}}_{\rm loc}(\R^3))}}+\|(H_{2\epsilon}-H_2)_i(H_1)_j\|_{\crm{L^{\frac{3}{2}}(0,T;L^{\frac{3}{2}}_{\rm loc}(\R^3))}}\\
&\leq\|(v_{2\epsilon})^\epsilon\|_{L^3(0,T;L^3_{\rm loc}(\R^3))}\|v_{2\epsilon}-v_2\|_{L^3(0,T;L^3_{\rm loc}(\R^3))}+\|v_{2\epsilon}\|_{L^3(0,T;L^3_{\rm loc}(\R^3))}\|(v_{2\epsilon})^\epsilon-v_2\|_{L^3(0,T;L^3_{\rm loc}(\R^3))}\\
&~~+\|(H_{2\epsilon})^\epsilon\|_{L^3(0,T;L^3_{\rm loc}(\R^3))}\|H_{2\epsilon}-H_2\|_{L^3(0,T;L^3_{\rm loc}(\R^3))}+\|H_{2\epsilon}\|_{L^3(0,T;L^3_{\rm loc}(\R^3))}\|(H_{2\epsilon})^\epsilon-H_2\|_{L^3(0,T;L^3_{\rm loc}(\R^3))}\\
&~~+\|v_1\|_{L^3(0,T;L^3_{\rm loc}(\R^3))}\|v_{2\epsilon}-v_2\|_{L^3(0,T;L^3_{\rm loc}(\R^3))}+\|H_1\|_{L^3(0,T;L^3_{\rm loc}(\R^3))}\|H_{2\epsilon}-H_2\|_{L^3(0,T;L^3_{\rm loc}(\R^3))}
\rightarrow 0
\end{align*}
as $\epsilon\rightarrow 0$. \crm{Let the integer $n>0$} such that the above local convergence holds in $\crm{L^{\frac{3}{2}}(B(2^n)\times(2^{-n},t))}$, then for \crm{any integer} $m>n\gg1$, it follows that
\begin{align*}
&\Pi^\epsilon-\Pi\\
 &\quad= -\frac{1}{3}\mathrm{tr}\big(N^{\epsilon}-N\big)
 +\int_{B(2)}K_{ij}(x-y)(\tilde N_{ij}^{\epsilon}
 -N_{ij})\mathrm{d}y\\
 &\quad\quad+ \left[\int_{B(2^{n+1})\setminus B(2)}+\int_{B(2^{m})\setminus B(2^{n+1})} \right][K_{ij}(x-y)-K_{ij}(-y)](\tilde N_{ij}^{\epsilon}-N_{ij})\mathrm{d}y\\
 &\quad\quad+\int_{\R^3\setminus B(2^{m})}[K_{ij}(x-y)-K_{ij}(-y)](\tilde N_{ij}^{\epsilon}
 -N_{ij})\mathrm{d}y\\
 &\quad\triangleq q_1+ q_2+ q_3+q_4+q_5.
 \end{align*}
\crm{Here, the matrix $N=(N_{ij})_{3\times3}$.}
By Calder\'{o}n-Zygmund inequality, one has
\begin{align}\label{1-22-1}
   \|q_1+ q_2+ q_3\|_{L^{\frac{3}{2}}L^{\frac{3}{2}}(2^{-n},t;B(2^n))}\leq \|\tilde N^{\epsilon}-N\|_{L^{\frac{3}{2}}L^{\frac{3}{2}}(2^{-n},t;B(2^{n+1}))},
 \end{align}
and
  \begin{align}\label{1-22-2}
   \|q_4\|_{L^{\frac{3}{2}}L^{\frac{3}{2}}(2^{-n},t;B(2^n))}\leq \|\tilde N^{\epsilon}-N\|_{L^{\frac{3}{2}}L^{\frac{3}{2}}(2^{-n},t;B(2^{m}))}.
 \end{align}
\crm{On the other hand,} notice that
$$ |K_{ij}(x-y)-K_{ij}(-y)|\leq \frac{|x|}{|y|^4},\quad\quad x\in B(2^n)$$
for any $\crm{y\in}\R^3\setminus B(2^{m}) \subset \bigcup_{x_0 \in \R^3\setminus B(2^{m})} B(x_0,1)\subset \R^3\setminus B(2^{m}-1)$, from which we can obtain
\crm{
\begin{align*}
&\Big\|\int_{\mathbb R^3\setminus B(2^m)}[K_{ij}(x-y)-K_{ij}(-y)]( N_{ij}^{\epsilon}-N_{ij})\mathrm{d}y\Big\|_{L^{\frac{3}{2}}(B(2^n)\times(2^{-n},t))}\\
&~~\leq \Big\|\int_{\bigcup_{x_0 \in \R^3\setminus B(2^{m})} B(x_0,1)}\frac{|x|}{|y|^4}
(N_{ij}^{\epsilon}-N_{ij})(y,s)\mathrm{d}y\Big\|
_{L^{\frac{3}{2}}(B(2^n)\times(2^{-n},t))}\\
&~~\leq 2^{3n}\Big\|\int_{\bigcup_{x_0 \in \R^3\setminus B(2^{m})} B(x_0,1)}\frac{1}{|y|^4}
(N_{ij}^{\epsilon}-N_{ij})(y,s)\mathrm{d}y\Big\|
_{L^{\frac{3}{2}}(2^{-n},t)}\\
&~~\leq  2^{3n}\Big\|\left(\int_{|y|>2^m-1}|y|^{-12}\mathrm{d}y
\right)^{\frac{1}{3}}\left(\int_{\R^3}
|N_{ij}^{\epsilon}-N_{ij}|^{\frac{3}{2}}(y,s)
\mathrm{d}y\right)^{\frac{2}{3}}\Big\|_{L^{\frac{3}{2}}(2^{-n},t)}\\
&~~\leq \frac{2^{3n}}{(2^m-1)^3}\Big(\|v\|^2_{L^{3}(\R^3\times(2^{-n},t))}
+\|H\|^2_{L^{3}(\R^3\times(2^{-n},t))}\Big).
 \end{align*}}
 Therefore, for any small \crm{$\varepsilon$}, there exist \crm{$m$,} depending only on \crm{$\varepsilon$} such that
\begin{align}\label{1-22-3}
\|q_5\|_{\crm{L^{\frac{3}{2}}(B(2^n)\times(2^{-n},t))}}< \crm{\varepsilon}.
\end{align}
By combining \eqref{1-22-1},\eqref{1-22-2} and \eqref{1-22-3}, one has
$$\lim_{\epsilon\rightarrow 0}\|\Pi^\epsilon-\Pi\|_{\crm{L^{\frac{3}{2}}(0,T; L_{\rm loc}^{\frac{3}{2}}(\R^3))}}=0.$$
Then for the pressure term $\int_{0}^{\infty}\int_{\mathbb R^3}\Pi^\epsilon v_{2\epsilon}\nabla\phi \,{\rm d}x{\rm d}s$, we can get
\begin{align*}
&\Big|\int_{0}^{\infty}\int_{\mathbb R^3}\Pi^\epsilon v_{2\epsilon}\nabla\phi \,{\rm d}x{\rm d}s-\int_{0}^{\infty}\int_{\mathbb R^3}\Pi v_2\nabla\phi \,{\rm d}x{\rm d}s\Big|\\
&\leq \|\nabla \phi\|_{\crm{L^\infty(\R^3\times(0,\infty))}}
\big(\|\Pi^\epsilon\|_{\crm{L^{\frac{3}{2}}(0,T; L_{\rm loc}^{\frac{3}{2}}(\R^3))}}\|v_{2\epsilon}-v_2\|_{\crm{L^{3}(0,T; L_{\rm loc}^{3}(\R^3))}}\\
&+\|\Pi^\epsilon-\Pi\|_{\crm{L^{\frac{3}{2}}(0,T; L_{\rm loc}^{\frac{3}{2}}(\R^3))}}\|v_2\|_{\crm{L^{3}(0,T; L_{\rm loc}^{3}(\R^3))}}\big)\rightarrow 0.
\end{align*}
\crm{To conclusion, the pair} $(v_2,H_2)$ satisfies the following local energy inequality
\begin{align*}\label{1-4-2}
&\int_{\mathbb{R}^3} \phi(|v_2|^2+|H_2|^2) \,{\rm d}x+2 \int_0^\infty \int_{\mathbb{R}^3} \phi(|\nabla v_2|^2+|\nabla H_2|^2) \,{\rm d}x \leq \sum_{i=1}^{7}I_i,
\end{align*}
\crm{for any $\phi\in C^\infty_0(\mathbb R^3\times(0,T))$. Here}
\begin{align*}
&I_1=\int_{0}^{\infty} \int_{\mathbb R^3}(|v_2|^2+|H_2|^2)\partial_t\phi\,\mathrm{d}x\mathrm{d}s,\\
&I_2=\int_{0}^{\infty}\int_{\mathbb R^3} \Big((|v_2|^2+|H_2|^2)\Delta\phi+(|v_2|^2+|H_2|^2)v_2\nabla\phi
-2v_2H_2(H_2\cdot\nabla\phi)\Big)\,{\rm d}x{\rm d}s,\\
&I_3=\int_{0}^{\infty}\int_{\mathbb R^3}\big(v_1\otimes v_2+v_2\otimes v_1+v_1\otimes v_1-H_1\otimes H_2-H_2\otimes H_1-H_1\otimes H_1\big):\nabla v_2\phi\,{\rm d}x{\rm d}s,\\
&I_4=\int_{0}^{\infty}\int_{\mathbb R^3}\big(H_1\otimes v_2+H_2\otimes v_1+H_1\otimes v_1-v_1\otimes H_2-v_2\otimes H_1-v_1\otimes H_1\big):\nabla H_2\phi\,{\rm d}x{\rm d}s,\\
&I_5=\int_{0}^{\infty}\int_{\mathbb R^3}\big(v_1\otimes v_2+v_2\otimes v_1+v_1\otimes v_1-H_1\otimes H_2-H_2\otimes H_1-H_1\otimes H_1\big):v_2\otimes\nabla\phi\,{\rm d}x{\rm d}s,\\
&I_6=\int_{0}^{\infty}\int_{\mathbb R^3}\big(H_1\otimes v_2+H_2\otimes v_1+H_1\otimes v_1-v_1\otimes H_2-v_2\otimes H_1-v_1\otimes H_1\big):H_2\otimes\nabla\phi\,{\rm d}x{\rm d}s,\\
&I_7=2\int_{0}^{\infty}\int_{\mathbb R^3}\Pi v_2\nabla\phi \,{\rm d}x{\rm d}s.
\end{align*}

\textbf{Step (iii):}~ {\em Prove that the global energy inequality \eqref{global energy inequality} holds.} To obtain the desired result, we take
$$\phi\triangleq\psi(t)\varphi_{R}(x),\quad where \quad \psi(t)\in{C}_{0}^{\infty}((0,T)),
\varphi_{R}(x)\in{C}_{0}^{\infty}(\mathbb R^3)$$
and $\varphi_{R}(x)$ is a smooth function such that
$$
\varphi_{R}(x)=\begin{cases}
1, & x \in B(R),\\
0, & x\in \R^3\backslash B(2R),
\end{cases}\quad |\nabla\varphi_{R}|\leq\frac{C}{R},\quad\text{and }\quad|\nabla^{2}\varphi_{R}
|\leq\frac{C}{R^{2}}.
$$
\crm{where $C$ is a constant independent of $R$.} \eqref{strong and weak convergence}  demonstrates that
\begin{equation}\label{v_2H_2estimates}
(v_2,H_2)\in L^\infty(0,T;L^2(\mathbb R^3)) \cap L^2(0,T;\dot{H}^1(\mathbb R^3)).
\end{equation}
This implies that
\begin{equation*}
\begin{aligned}
&\phi(|v_2|^2+|H_2|^2)\leq \psi(|v_2|^2+|H_2|^2)\in L^1(\R^3)\\
&\phi(|v_2|^2+|H_2|^2)\rightarrow \psi(|v_2|^2+|H_2|^2),\quad\text{as}\quad R\rightarrow\infty.
\end{aligned}
\end{equation*}
By the Dominated Convergence Theorem, one has
\begin{align*}
\lim_{R\rightarrow\infty}\int_{\mathbb{R}^{3}}(|v_2|^2+|H_2|^2)\phi\mathrm{d}x
     =\int_{\mathbb{R}^{3}}(|v_2|^2+|H_2|^2)\psi(t)\mathrm{d}x.
\end{align*}
Similarly,
\begin{align*}
\lim_{R\rightarrow\infty}\int_{0}^{\infty}\int_{\mathbb{R}^{3}}(|\nabla v_2|^2+|\nabla H_2|^2)\phi\mathrm{d}x\mathrm{d}s=
     \int_{0}^{\infty}\int_{\mathbb{R}^{3}}(|\nabla v_2|^2+|\nabla H_2|^2)\psi(s)\mathrm{d}x\mathrm{d}s.
     \end{align*}
For $I_1$, similarly, based on \eqref{v_2H_2estimates} and the Dominated Convergence Theorem, we have
$$\lim_{R\rightarrow\infty}I_1=\int_{0}^{\infty} \int_{\mathbb R^3}(|v_2|^2+|H_2|^2)\partial_t\psi\mathrm{d}x\mathrm{d}s.$$
For $I_3,I_4$, according to \eqref{v_2H_2estimates}, Remark \ref{remark} and H$\mathrm{\ddot{o}}$lder's inequalities, we derive that
\begin{equation*}
\begin{aligned}
\lim_{R\rightarrow\infty}(I_3+I_4)&=\int_{0}^{\infty}\int_{\mathbb R^3}\big(v_1\otimes v_2+v_2\otimes v_1+v_1\otimes v_1\\
&-H_1\otimes H_2-H_2\otimes H_1-H_1\otimes H_1\big):\nabla v_2\psi\,{\rm d}x{\rm d}s\\
&+\int_{0}^{\infty}\int_{\mathbb R^3}\big(H_1\otimes v_2+H_2\otimes v_1+H_1\otimes v_1\\
&-v_1\otimes H_2-v_2\otimes H_1-v_1\otimes H_1\big):\nabla H_2\psi\,{\rm d}x{\rm d}s.
\end{aligned}
\end{equation*}
Now we claim $\lim_{R\to\infty}I_2=0$. In fact, it is clear that from \eqref{v_2H_2estimates},
\begin{equation*}
\begin{aligned}
&\lim_{R\rightarrow\infty}\Big|\int_0^\infty\int_{\R^3}(|v_2|^2
+|H_2|^2)\psi(t)\Delta\varphi_R\mathrm{d}x\mathrm{d}s\Big|\leq \lim_{R\rightarrow\infty}\frac{C}{R^2}=0
\end{aligned}
\end{equation*}
and
\begin{equation*}
\begin{aligned}
&\lim_{R\rightarrow\infty}\big|\int_{0}^{\infty}\int_{\mathbb R^3}(|v_2|^2+|H_2|^2)v_2\crm{\psi(t)\nabla\varphi_R} \,{\rm d}x{\rm d}s\big|\\
&\leq \lim_{R\rightarrow\infty}\frac{c}{R}\big(\|v_2\|^3_{L^3(0,T;L^3(\mathbb R^3))}+\|H_2\|^2_{L^3(0,T;L^3(\mathbb R^3))}\|v_2\|_{L^3(0,T;L^3(\mathbb R^3))}\big)=0.
\end{aligned}
\end{equation*}
Similar, by the \eqref{v_2H_2estimates}
$$\lim_{R\rightarrow\infty}\big|\int_{0}^{\infty}\int_{\mathbb R^3}v_2H_2(H_2\cdot\nabla\phi)\,{\rm d}x{\rm d}s\big|=0,$$
so we conclude the desired result.
Based on Remark \ref{remark} and \eqref{v_2H_2estimates}, $I_5,I_6$ can be handled in the same way, from which we obtain
$$\lim_{R\rightarrow\infty}(I_5+I_6)=0.$$
\crm{Next, we turn to consider the term $I_7$, and} divide it into four parts
\begin{equation*}
	\begin{aligned}
		I_7=\int_{0}^{\infty}\int_{\mathbb R^3}\Pi v_2\nabla\phi \,{\rm d}x{\rm d}s=\sum_{j=1}^4I_{7j},
	\end{aligned}
\end{equation*}
where
 $$I_{7j}=\int_{0}^{\infty}\int_{\mathbb R^3}v_2\nabla\phi(\Pi^j-[\Pi^j]_{B(2R)})\,{\rm d}x{\rm d}s.$$
 Thus, by the Lemma \ref{pressure estimate}, one has

 \begin{equation*}
 	\begin{aligned}
 		|I_{71}|&\leq \frac{C}{R}\|v_2\|_{\crm{L^3(B(2R)\times(0,T))}}\Big(\crm{\int_{0}^{T}
 \int_{B(2R)}}|\Pi^1-[\Pi^1]_{\crm{B(2R)}}|^{\frac{3}{2}}\,{\rm d}x{\rm d}s\Big)^{\frac{2}{3}}\\
 		&\leq \frac{C}{R}\|v_2\|_{\crm{L^3(B(2R)\times(0,T))}}R^{\frac{1}{3}}
 \Big(\crm{\int_{0}^{T}
 (\int_{B(2R)}}|\nabla\Pi^1|^\frac{9}{8}\,{\rm d}x)^\frac{4}{3}\,{\rm d}s\Big)^\frac{2}{3}\rightarrow 0,
 	\end{aligned}
 \end{equation*}
\begin{equation*}
	\begin{aligned}
		|I_{72}|&\leq \frac{C}{R}\|v_2\|_{\crm{L^3(B(2R)\times(0,T))}}\Big(\crm{\int_{0}^{T}\int_{B(2R)}}
|\Pi^2-[\Pi^2]_{\crm{B(2R)}}|^{\frac{3}{2}}\,{\rm d}x{\rm d}s\Big)^{\frac{2}{3}}\\
		&\leq \frac{C}{R}\|v_2\|_{\crm{L^3(B(2R)\times(0,T))}}R^\frac{3}{4}
\Big(\crm{\int_{0}^{T}(\int_{B(2R)}}|\nabla\Pi^2|^\frac{4}{3}\,{\rm d}x)^\frac{9}{8}\,{\rm d}s\Big)^\frac{2}{3}\rightarrow 0,
	\end{aligned}
\end{equation*}
\begin{equation*}
	\begin{aligned}
		|I_{73}|&\leq \frac{C}{R}\|v_2\|_{L^3(B(2R)\times(0,T))}\Big(\int_{0}^{T}\int_{B(2R)}
|\Pi^3-[\Pi^3]_{B(2R)}|^{\frac{3}{2}}\,{\rm d}x{\rm d}s\Big)^{\frac{2}{3}}\\
		&\leq \frac{C}{R}\|v_2\|_{L^3(B(2R)\times(0,T))}R^\frac{1}{2}
\Big(\int_{0}^{T}(\int_{B(2R)}|\nabla\Pi^3|^\frac{6}{5}\,{\rm d}x)^\frac{5}{4}\,{\rm d}s\Big)^\frac{2}{3}\rightarrow 0,
	\end{aligned}
\end{equation*}
and
\begin{equation*}
\begin{aligned}
|I_{74}|&\leq \frac{C}{R}\|v_2\|_{L^3(B(2R)\times(0,T))}\Big(\int_{0}^{T}\int_{B(2R)}
|\Pi^4-[\Pi^4]_{B(2R)}|^{\frac{3}{2}}\,{\rm d}x{\rm d}s\Big)^{\frac{2}{3}}\\
&\leq \frac{C}{R}\|v_2\|_{L^3(B(2R)\times(0,T))}R^\frac{3}{4}
\Big(\int_{0}^{T}\int_{B(2R)}|\nabla\Pi^4|^\frac{3}{2}\,{\rm d}x{\rm d}s\Big)^\frac{2}{3}\rightarrow 0,
\end{aligned}
\end{equation*}
as $R\rightarrow\infty$, then one has $$\lim_{R\rightarrow\infty}I_7=0.$$
By combining the above processes, we can obtain
\begin{equation}\label{1-4-3}
\int_{\mathbb{R}^{3}}(|v_2|^2+|H_2|^2)\psi(t)\mathrm{d}x+
  2\int_{0}^{\infty}\int_{\mathbb{R}^{3}}(|\nabla v_2|^2+|\nabla H_2|^2)\psi(s)\mathrm{d}x\mathrm{d}s\leq I_8+ I_9+I_{10},
\end{equation}
where
\begin{align*}
I_{8}&=\int_{0}^{\infty}\int_{\mathbb{R}^{3}}\Big(|v_2|^2+|H_2|^2\Big)\partial_{s}\psi
\mathrm{d}x\mathrm{d}s,\\
I_9&=\int_{0}^{\infty}\int_{\mathbb R^3}\big(v_1\otimes v_2+v_2\otimes v_1+v_1\otimes v_1-H_1\otimes H_2-H_2\otimes H_1-H_1\otimes H_1\big):\nabla v_2\psi\,{\rm d}x{\rm d}s,\\
I_{10}&=\int_{0}^{\infty}\int_{\mathbb R^3}\big(H_1\otimes v_2+H_2\otimes v_1+H_1\otimes v_1-v_1\otimes H_2-v_2\otimes H_1-v_1\otimes H_1\big):\nabla H_2\psi\,{\rm d}x{\rm d}s.
\end{align*}
Besides, we take
$\psi(s)=\psi_{\varepsilon}(s)= \crm{\chi_{t}}(s/\varepsilon)$  in \eqref{1-4-3},
 where \crm{$t>3\varepsilon$ and} $\crm{\chi_t}(s)$ is a smooth function such that $0\leq \crm{\chi_{t}} \leq 1$,
$$
\crm{\chi_{t}}(s)=\begin{cases}
0, & s\leq1,\\
1, & 2\leq s\leq \frac{t-\varepsilon}{\varepsilon},\\
0, & s\geq \frac{t}{\varepsilon}.
\end{cases}
$$
It is clear that
$$
 \int_1^2\partial_s\crm{\chi_{t}}(s)\mathrm{d}s=1 \text{ and }\ \int_{\frac{t-\varepsilon}{\varepsilon}}^{\frac{t}{\varepsilon}}\partial_s\crm{\chi_{t}}(s)\mathrm{d}s=-1.
$$
Thus,
\EQS{\label{1-4-4}
\psi_{\varepsilon}(s)=
\begin{cases}
{0 ,}&\quad {s \leq \varepsilon },\\[2mm]
{1 ,}&\quad {2\varepsilon \leq s\leq t-\varepsilon},\\[2mm]
{0 ,}&\quad {s \geq t }
\end{cases}\ \int_{\varepsilon}^{2\varepsilon}\partial_s\psi_{\varepsilon}(s)
\mathrm{d}s=1 \text{ and }\ \int_{t-\varepsilon}^{t}\partial_s\psi_{\varepsilon}(s)
\mathrm{d}s=-1.
}
Notice that $\psi(t)=0$, therefore,
$$\int_{\mathbb{R}^{3}}(|v_2|^2+|H_2|^2)\psi(t)\mathrm{d}x=0.$$
From \eqref{v_2H_2estimates}, one has
\begin{equation*}
\begin{aligned}
&\psi(s)\int_{\mathbb{R}^{3}}(|\nabla v_2|^2+|\nabla H_2|^2)\mathrm{d}x\leq\int_{\mathbb{R}^{3}}(|\nabla v_2|^2+|\nabla H_2|^2)\mathrm{d}x\in L^1(0,t),\\
&\psi(s)\int_{\mathbb{R}^{3}}(|\nabla v_2|^2+|\nabla
H_2|^2)\mathrm{d}x\rightarrow \int_{\mathbb{R}^{3}}(|\nabla v_2|^2+|\nabla H_2|^2),\quad\text{as}\quad \varepsilon\rightarrow 0.
\end{aligned}
\end{equation*}
Then, by the Dominated Convergence Theorem, we have
\begin{align*}
\lim_{\varepsilon\rightarrow0}\int_{0}^{\infty}\int_{\mathbb{R}^{3}}
(|\nabla v_2|^2+|\nabla
H_2|^2))\psi_\varepsilon(s)
\mathrm{d}x\mathrm{d}s=\crm{\int_{0}^{t}}\int_{\mathbb{R}^{3}}
(|\nabla v_2|^2+|\nabla
H_2|^2)\mathrm{d}x\mathrm{d}s.
\end{align*}
\crm{For $I_9$ and $I_{10}$, notice that}
\begin{equation*}
\begin{aligned}
&\crm{\int_{0}^{t}}\int_{\mathbb R^3}|v_1\otimes v_2:\nabla v_2|\,{\rm d}x{\rm d}s\\
&\leq\crm{\int_{0}^{t}}\|v_1\|_{L^5(\mathbb R^3)}\|v_2\|_{L^{\frac{10}{3}}(\mathbb R^3)}\|\nabla v_2\|_{L^2(\mathbb R^3)}\mathrm{d}s,\\
&\leq\crm{\int_{0}^{t}}\|v_1\|_{L^5(\mathbb R^3)}\|v_2\|^{\frac{2}{5}}_{L^2(\mathbb R^3)}\|\nabla v_2\|^{\frac{8}{5}}_{L^2(\mathbb R^3)}\mathrm{d}s\\
&\leq \|v_1\|_{L^5(0,t;L^5(\mathbb R^3))}\|v_2\|^{\frac{2}{5}}_{L^\infty(0,t;L^2(\mathbb R^3))}\|\nabla v_2\|^{\frac{8}{5}}_{L^2(0,t;L^2(\mathbb R^3))}.
\end{aligned}
\end{equation*}
\crm{The remaining terms of $I_9$ and $I_{10}$ can be estimated as above.  Then, by the Dominated Convergence Theorem, we have}
\begin{align*}
\lim_{\varepsilon\rightarrow0}&I_9=\crm{\int_{0}^{t}}\int_{\mathbb R^3}\big(v_1\otimes v_2+v_2\otimes v_1+v_1\otimes v_1-H_1\otimes H_2-H_2\otimes H_1-H_1\otimes H_1\big):\nabla v_2\,{\rm d}x{\rm d}s,\\
\lim_{\varepsilon\rightarrow0}&I_{10}=\crm{\int_{0}^{t}}\int_{\mathbb R^3}\big(H_1\otimes v_2+H_2\otimes v_1+H_1\otimes v_1-v_1\otimes H_2-v_2\otimes H_1-v_1\otimes H_1\big):\nabla H_2\,{\rm d}x{\rm d}s.
\end{align*}
\crm{
Finally, we turn to investigate the term $I_8$. To do this, we can rewrite the term $I_8$ as
\EQs{
&\int_{0}^{\infty}
\int_{\mathbb{R}^{3}}(|v_2(x,s)|^{2}+|H_2(x,s)|^{2})
\partial_{s}\psi_\varepsilon(s)\,\mathrm{d}x\mathrm{d}s\\
&=
\int_{\varepsilon}^{2\varepsilon}\int_{\R^3}(|v_2(x,s)|^{2}+|H_2(x,s)|^{2})
\partial_{s}\psi_\varepsilon(s)\,\mathrm{d}x\mathrm{d}s\\
&~~~~+\int_{t-\varepsilon}^{t}\int_{\R^3}(|v_2(x,s)|^{2}+|H_2(x,s)|^{2})
\partial_{s}\psi_\varepsilon(s)\,\mathrm{d}x\mathrm{d}s,
}
due to the fact $\partial_{s}\psi_\varepsilon(s)=0$ for $s\in (2\varepsilon,t-\varepsilon)$.}
It is easy to see that by the fact $|\partial_s\psi_\varepsilon|\leq \frac{C}{\varepsilon}$
\EQs{
&\left|\int_{\varepsilon}^{2\varepsilon}\int_{\R^3}(|v_2(x,s)|^{2}+|H_2(x,s)|^{2})
\partial_{s}\psi_\varepsilon(s)\,\mathrm{d}x\mathrm{d}s\right|\\
&\leq\frac{C}{\varepsilon}\left|\int_{\varepsilon}^{2\varepsilon}\int_{\R^3}|v_2(x,s)|^{2}+|H_2(x,s)|^{2}
\mathrm{d}x\mathrm{d}s\right|\\&
\to \int_{\R^3}|v_2(\cdot,0)|^{2}+|H_2(\cdot,0)|^{2}
\mathrm{d}x=0\quad\quad \mbox{as}\quad \varepsilon\to0.
}
On the other hand, from \eqref{1-4-4}, we have for a.e $t\in (0,T)$
\EQs{
&\lim_{\varepsilon\to0}\int_{t-\varepsilon}^{t}\int_{\R^3}(|v_2(x,s)|^{2}+|H_2(x,s)|^{2})
\partial_{s}\psi_\varepsilon(s)\,\mathrm{d}x\mathrm{d}s\\
&=
\lim_{\varepsilon\to0}\int_{t-\varepsilon}^{t}\int_{\R^3}\big\{(|v_2(x,s)|^{2}+|H_2(x,s)|^{2})-
(|v_2(x,t)|^{2}+|H_2(x,t)|^{2})\big\}\partial_{s}
\psi_\varepsilon(s)\,\mathrm{d}x\mathrm{d}s\\
&~~~~-\int_{\R^3}|v_2(x,t)|^{2}+|H_2(x,t)|^{2}\,\mathrm{d}x\\&
=-\int_{\R^3}|v_2(x,t)|^{2}+|H_2(x,t)|^{2}\,\mathrm{d}x.
}
Here, we have used the fact as $\varepsilon\to0$
\EQs{
&\left|\int_{t-\varepsilon}^{t}\int_{\R^3}\big\{(|v_2(x,s)|^{2}+|H_2(x,s)|^{2})-
(|v_2(x,t)|^{2}+|H_2(x,t)|^{2})\big\}\partial_{s}\psi_\varepsilon(s)\,\mathrm{d}x\mathrm{d}s\right|\\&
\leq\frac{C}{\varepsilon}\int_{t-\varepsilon}^{t}\int_{\R^3}\big\{(|v_2(x,s)|^{2}+|H_2(x,s)|^{2})-
(|v_2(x,t)|^{2}+|H_2(x,t)|^{2})\big\}\,\mathrm{d}x\mathrm{d}s\to0.
}
Thus, we conclude
\begin{align*}
\lim_{\varepsilon\rightarrow0}I_8
=-\int_{\mathbb{R}^{3}}|v_2(x,t)|^{2}+|H_2(x,t)|^{2}\,\mathrm{d}x.
\end{align*}
Combining the above discussion, we show the global energy inequality \eqref{global energy inequality} holds.
\end{proof}
\section{Uniqueness}
In this section, we give a proof of Theorem \ref{thm 1.2}. The tool we use in this section is the Gronwall's inequality. To begin with, we present the following
lemma which plays a key role in our proof Theorem \ref{thm 1.2}.
\begin{lem}[\cite{A. Mahalov}]\label{lem1-13}
Assume that $(v, H) \in L^\infty L^2(Q_R(z_0))\cap L^2H^1 (Q_R(z_0))$ and $\Pi\in L_{\frac{3}{2}}(Q_R(z_0))$ satisfy \eqref{EQ-tlh} in the sense of distributions, where $Q_R(z_0)=(t_0-R^2,t_0)\times B(x_0,R)$. Assume, in addition, that
$$ (v, H) \in L^\infty L^3(Q_R(z_0)),$$
then $(v,H)$ is H\"{o}lder continuous on $\overline{Q}_{\frac{1}{2}R}(z_0)$.
\end{lem}

With the above lemma, we can now proceed with the proof of Theorem \ref{thm 1.2}
\begin{proof}[Proof of Theorem \ref{thm 1.2}]
To achieve the uniqueness of the solution, we first consider regularity of the solution. From Lemma \ref{lem1-13}, we can deduce that  for sufficiently small $R$, if $t_0-R^2>0$ and $t_0<T$, then
$$
\|(v,H)\|_{L^\infty(Q_{\frac{1}{2}R}(z_0))}\leq C(R).
$$
Thus, $z_0=(x_0,t_0)$ is a regular point of $(v,H)$, which implies
 $(v,H)\in L^\infty(\delta,T;L^\infty(\mathbb R^3))$ for any $\delta>0$.
Next, we claim that
$$
(v_2,H_2)\in W^{2,1}_2(\delta,T;\mathbb R^3),\,(\nabla v_2,\nabla H_2)\in L^\infty (\delta,T;L^2(\mathbb R^3)),
$$
and
$$\nabla\Pi\in L^2(\delta,T;L^2(\mathbb R^3)).$$
Indeed, by Lemma \ref{Solonnikove estimates} and heat kernel estimates, one has
\begin{align*}
		\int_{\delta}^{T}\|\partial_t v_2\|^2_{L^2(\mathbb R^3)}\,{\rm d}t+\int_{\delta}^{T}\|\nabla^2 v_2\|^2_{L^2(\mathbb R^3)}\,{\rm d}t&+\int_{\delta}^{T}\|\nabla\Pi\|^2_{L^2(\mathbb R^3)}\,{\rm d}t\\
&\leq C	\int_{\delta}^{T}\|H\cdot\nabla H-v\cdot\nabla v\|^2_{L^2(\mathbb R^3)}\,{\rm d}t,
\end{align*}
\begin{align*}
\int_{\delta}^{T}\|\partial_t H_2\|^2_{L^2(\mathbb R^3)}\,{\rm d}t+\int_{\delta}^{T}\|\nabla^2 H_2\|^2_{L^2(\mathbb R^3)}\,{\rm d}t\leq C	\int_{\delta}^{T}\|H\cdot\nabla v-v\cdot\nabla H\|^2_{L^2(\mathbb R^3)}\,{\rm d}t.	\end{align*}
Then by the heat kernel estimates, we can obtain
\begin{equation*}\label{vnablav}
\begin{aligned}
   &\int_{\delta}^{T}\|v\cdot\nabla v\|_{L^{2}(\mathbb{R}^{3})}^{2}\mathrm{d}t\\
   &\leq \int_{\delta}^{T}\|v\cdot\nabla v_1\|_{L^{2}(\mathbb{R}^{3})}^{2}\mathrm{d}t+ \int_{\delta}^{T}\|v\cdot\nabla v_2\|_{L^{2}(\mathbb{R}^{3})}^{2}\mathrm{d}t\\
   &\leq\int_{\delta}^{T}\|v\|_{L^{4}(\mathbb{R}^{3})}^{2}\mathrm{d}t+ \int_{\delta}^{T}\|\nabla v_1\|_{L^{4}(\mathbb{R}^{3})}^{2}\mathrm{d}t+ \|v\|^{2}_{L^\infty(\delta,T;\mathbb R^3)}\int_{\delta}^{T}\|\nabla v_2\|_{L^{2}(\mathbb{R}^{3})}^{2}\mathrm{d}t\\
   &\leq\int_{\delta}^{T}\big(\|v_1\|_{L^{4}(\mathbb{R}^{3})}^{2}+\|v_2\|_{L^{4}(\mathbb{R}^{3})}^{2}\big)\mathrm{d}t+ \int_{\delta}^{T}\|\nabla v_1\|_{L^{4}(\mathbb{R}^{3})}^{2}\mathrm{d}t\\
   &~~~~~~~+ \|v\|^{2}_{L^\infty(\delta,T;\mathbb R^3)}\int_{\delta}^{T}\|\nabla v_2\|_{L^{2}(\mathbb{R}^{3})}^{2}\mathrm{d}t\\
   &\leq C\int_{\delta}^{T}(t^{-\frac{1}{12}}+t^{-\frac{5}{4}})\|v_0\|^2_{L^3(\mathbb R^3)}\mathrm{d}t+\int_{\delta}^{T}\|v_2\|_{L^{2}(\mathbb{R}^{3})}^{2}\mathrm{d}t\\
   &~~~~~~~+\big(\|v\|^{2}_{L^\infty(\delta,T;\mathbb R^3)}+1\big)\int_{\delta}^{T}\|\nabla v_2\|_{L^{2}(\mathbb{R}^{3})}^{2}\mathrm{d}t<\infty.
\end{aligned}
\end{equation*}
Similarly, we also have
$$H\cdot\nabla H,H\cdot\nabla v,v\cdot\nabla H \in L^{2}(\delta,T;\mathbb{R}^{3}).$$
Thus,
$$\partial_t v_2,~\partial_t H_2,~\nabla^2 v_2,~\nabla^2 H_2,~\nabla\Pi \in L^{2}(\delta,T;\mathbb R^3).$$
 Next, we multiply $\eqref{v_2}_1$ by $-\Delta v_2$ and $\eqref{v_2}_2$ by $-\Delta H_2$, then integrate over $\mathbb R^3$ to obtain
\begin{equation*}
	\begin{aligned}
		&\frac{1}{2}\frac{d}{dt}\big(\|\nabla v_2\|^2_{L^2(\mathbb R^3)}+\|\nabla H_2\|^2_{L^2(\mathbb{R}^{3})}\big)+\|\Delta v_2\|^2_{L^2(\mathbb R^3)}+\|\Delta H_2\|^2_{L^2(\mathbb R^3)}\\
		&\leq\|\Delta v_2\|_{L^2(\mathbb R^3)}\big(\|v\cdot\nabla v\|_{L^{2}(\mathbb{R}^{3})}+\|H\cdot\nabla H\|_{L^{2}(\mathbb{R}^{3})}\big)\\
&~~+\|\Delta H_2\|_{L^2(\mathbb R^3)}\big(\|v\cdot\nabla H\|_{L^{2}(\mathbb{R}^{3})}+\|H\cdot\nabla v\|_{L^{2}(\mathbb{R}^{3})}\big)\\
		&\leq \|\Delta v_2\|^2_{L^2(\mathbb R^3)}+\|\Delta H_2\|^2_{L^2(\mathbb R^3)}+\|\nabla v_2\|^2_{L^2(\mathbb R^3)}+\|\nabla H_2\|^2_{L^2(\mathbb R^3)}\\
&+\|v\cdot\nabla v_1\|^2_{L^{2}(\mathbb{R}^{3})}+\|H\cdot\nabla H_1\|^2_{L^{2}(\mathbb{R}^{3})}+\|v\cdot\nabla H_1\|^2_{L^2(\mathbb{R}^{3})}+\|H\cdot\nabla v_1\|^2_{L^2(\mathbb{R}^{3})},
	\end{aligned}
\end{equation*}
from which we further obtain
\begin{equation*}
\begin{aligned}
&\frac{1}{2}\frac{d}{dt}\Big(\|\nabla v_2\|^2_{L^2(\mathbb R^3)}+\|\nabla H_2\|^2_{L^2(\mathbb R^3)}\Big)\\
&\leq \|\nabla H\|^2_{L^2(\mathbb R^3)}+\|\nabla v\|^2_{L^2(\mathbb R^3)}+\|v\cdot\nabla v_1\|^2_{L^{2}(\mathbb{R}^{3})}\\
&+\|H\cdot\nabla H_1\|^2_{L^{2}(\mathbb{R}^{3})}+\|v\cdot\nabla H_1\|^2_{L^2(\mathbb{R}^{3})}+\|H\cdot\nabla v_1\|^2_{L^2(\mathbb{R}^{3})}.
\end{aligned}
\end{equation*}
This together with the Lemma \ref{Gronwall inequality} implies $\nabla v_2,\nabla H_2\in L^\infty (\delta,T;L^2(\mathbb R^3))$.

Next, we consider the uniqueness of the solution $(v,H)$. We first multiply $\eqref{v_2}_1$ by $v_2$ and $\eqref{v_2}_2$ by $H_2$, then integrate over $(0,t)\times \mathbb R^3$ to obtain
\begin{equation*}
\begin{aligned}
&\frac{1}{2}\int_{\mathbb R^3}|v_2(x,t)|^2+|H_2(x,t)|^2,{\rm d}x+\int_{0}^{t}\int_{\mathbb R^3}\big(|\nabla v_2|^2+|\nabla H_2|^2\big)\,{\rm d}x{\rm d}s\\
&=\int_{0}^{t}\int_{\mathbb R^3}(v\otimes v-H\otimes H):\nabla v_2+(H\otimes v-v\otimes H):\nabla H_2\,{\rm d}x{\rm d}s,
\end{aligned}
\end{equation*}
holds for any $t\in (0,T)$. Similarly, multiplying $\eqref{v_2}_1$ by $\tilde v_2$ and $\eqref{v_2}_2$ by $\tilde H_2$, then integrating over $(0,t)\times\mathbb R^3$, one has
\begin{align*}
&\int_{0}^{t}\int_{\mathbb R^3}\big(\tilde v_2\cdot\partial_t v_2+\tilde H_2\cdot\partial_t H_2\big)\,{\rm d}x{\rm d}s+\int_{0}^{t}\int_{\mathbb R^3}\big(\nabla v_2:\nabla \tilde v_2+\nabla H_2:\nabla\tilde H_2 \big)\,{\rm d}x{\rm d}s\\
&+\int_{0}^{t}\int_{\mathbb R^3}\big(H\otimes H-v\otimes v\big):\nabla \tilde v_2+\big(H\otimes v-v\otimes H\big):\nabla\tilde H_2\,{\rm d}x{\rm d}s=0.
\end{align*}
Let $\omega=\tilde v_2-v_2,~~\psi=\tilde H_2-H_2$, it is easy to derive
\begin{equation}\label{1-12-1}
\begin{aligned}
		&\int_{0}^{t}\int_{\mathbb R^3}\big(\tilde v_2\cdot\partial_t v_2+\tilde H_2\cdot\partial_t H_2 \big)\,{\rm d}x{\rm d}s-\frac{1}{2}\int_{\mathbb R^3}\big(|v_2(x,t)|^2+|H_2(x,t)|^2\big)\,{\rm d}x\\
&+\int_{0}^{t}\int_{\mathbb R^3}
\big(H\otimes H:-v\otimes v\big):\nabla w+\big(H\otimes v:-v\otimes H\big):\nabla \psi\,{\rm d}x{\rm d}s\\
&+\int_{0}^{t}\int_{\mathbb R^3}
\big(\nabla v_2:\nabla\omega+\nabla H_2\nabla\psi\big)\,{\rm d}x{\rm d}s=0.
\end{aligned}
\end{equation}
Besides, we also have
\begin{equation}\label{1-12-2}
\begin{aligned}
&\int_{\mathbb R^3}\big(\tilde v_2(x,t)\cdot v_2(x,t)+\tilde H_2(x,t)\cdot H_2(x,t)\big)\,{\rm d}x-\int_{0}^{t}\int_{\mathbb R^3}\big(\tilde v_2\cdot\partial_t v_2+\tilde H_2\cdot\partial_t H_2\big)\,{\rm d}x{\rm d}s\\
&+\int_{0}^{t}\int_{\mathbb R^3}\Big(\nabla\tilde v_2:\nabla v_2+\nabla\tilde H_2:\nabla H_2+\big(\tilde H\otimes\tilde H-\tilde v\otimes\tilde v\big):\nabla v_2\\
&+\big(\tilde H\otimes\tilde v-\tilde v\otimes\tilde H\big):\nabla H_2\Big)\,{\rm d}x{\rm d}s=0.
\end{aligned}
\end{equation}
From Definition \ref{def1}, we know that $(\tilde v_2,\tilde H_2)$ satisfies
\begin{equation}\label{1-12-3}
\begin{aligned}
&\frac{1}{2}\int_{\mathbb R^3}\big(|\tilde v_2(x,t)|^2+|\tilde H_2(x,t)|^2\big)\,{\rm d}x+\int_{0}^{t}\int_{\mathbb R^3}\big(|\nabla\tilde v_2|^2+|\nabla\tilde H_2|^2\big)\,{\rm d}x{\rm d}s\\
&\leq \int_{0}^{t}\int_{\mathbb R^3}\Big(\big(\tilde v\otimes\tilde v-\tilde H\otimes\tilde H\big):\nabla\tilde v_2+\big(\tilde v\otimes\tilde H-\tilde H\otimes\tilde v\big):\nabla\tilde H_2\Big)\,\crm{{\rm d}x{\rm d}s}.
\end{aligned}
\end{equation}
Combining \eqref{1-12-1}-\eqref{1-12-3}, one has
\begin{equation*}
\begin{aligned}
&\frac{1}{2}\int_{\mathbb R^3}\big(|\omega(x,t)|^2+|\psi(x,t)|^2\big)\,{\rm d}x+\int_{0}^{t}\int_{\mathbb R^3}\big(|\nabla\omega|^2+|\nabla\psi|^2\big)\,{\rm d}x{\rm d}s\\
		&\leq \int_{0}^{t}\int_{\mathbb R^3}\big(\omega\otimes v+v\otimes \omega-\psi\otimes H-H\otimes\psi\big):\nabla \omega\,{\rm d}x{\rm d}s\\
		&\quad+\int_{0}^{t}\int_{\mathbb R^3}\big(\psi\otimes v+H\otimes \omega-\omega\otimes H-v\otimes\psi\big):\nabla \omega\,{\rm d}x{\rm d}s\\
		&\leq 2\int_{0}^{t}\int_{\mathbb R^3}\big(|\omega||v||\nabla \omega|+|\psi||H||\nabla \omega|+|\psi||v||\nabla\psi|+|\omega||H||\nabla\psi|\big)\,{\rm d}x{\rm d}s.
	\end{aligned}
\end{equation*}
By the \crm{H\"{o}lder's} inequality and Young's inequality, we can obtain
\begin{equation}\label{inequality}
	\begin{aligned}
		&\int_{\mathbb R^3}\big(|\omega(x,t)|^2+|\psi(x,t)|^2\big)\,{\rm d}x+\int_{0}^{t}\int_{\mathbb R^3}\big(|\nabla \omega|^2+|\nabla \psi|^2\big)\,{\rm d}x{\rm d}s\\
		&\leq \crm{C}\int_{0}^{t}\int_{\mathbb R^3}(|\omega|^2+|\psi|^2)(|v|^2+|H|^2)\,{\rm d}x{\rm d}s.
	\end{aligned}
\end{equation}
The right-hand side of \eqref{inequality} contains four terms. For the term $\int_{0}^{t}\int_{\mathbb R^3}|\omega|^2|v|^2\,{\rm d}x{\rm d}s$, one has
\begin{align*}
\int_{0}^{t}\int_{\mathbb R^3}|\omega|^2|v|^2\,{\rm d}x{\rm d}s&\leq \crm{C}\int_{0}^{t}\int_{\mathbb R^3}|\omega|^2|v_1|^2\,{\rm d}x{\rm d}s+\crm{C\int_{0}^{t}\int_{\mathbb R^3}|\omega|^2|v_2|^2\,{\rm d}x{\rm d}s}\\
&=\crm{C I_1+C I_2}.
\end{align*}
For $I_1$, we have the following estimate:
\begin{equation*}
	\begin{aligned}
		I_1&\leq \int_{0}^{t}\Big(\int_{\mathbb R^3}|v_1|^5\,{\rm d}x\Big)^{\frac{2}{5}}\Big(\int_{\mathbb R^3}|\omega|^{2\cdot\frac{5}{3}}\,{\rm d}x\Big)^{\frac{3}{5}}\,{\rm d}s\\
&=\int_{0}^{t}\Big(\int_{\mathbb R^3}|v_1|^5\,{\rm d}x\Big)^{\frac{2}{5}}\Big(\int_{\mathbb R^3}|\omega|^{\frac{10}{3}}\,{\rm d}x\Big)^{\frac{3}{5}}\,{\rm d}s\\
		&\leq C\Big(\int_{0}^{t}\int_{\mathbb R^3}|v_1(y,t)|^5\,{\rm d}y\int_{\mathbb R^3}|\omega(x,t)|^2\,{\rm d}x{\rm d}s\Big)^\frac{2}{5}\Big(\int_{0}^{t}\int_{\mathbb R^3}|\nabla \omega|^2 \,{\rm d}x{\rm d}s\Big)^\frac{3}{5}\\
		&\leq C\Big(\int_{0}^{t}g_1(s)\int_{\mathbb R^3}|\omega(x,t)|^2\,{\rm d}x{\rm d}s\Big)^\frac{2}{5}\Big(\int_{0}^{t}\int_{\mathbb R^3}|\nabla \omega|^2\,{\rm d}x{\rm d}s\Big)^\frac{3}{5},\\
&\leq c_1\int_{0}^{t}g_1(s)\int_{\mathbb R^3}|\omega(x,t)|^2\,{\rm d}x{\rm d}s+\epsilon\int_{0}^{t}\int_{\mathbb R^3}|\nabla \omega|^2\,{\rm d}x{\rm d}s
	\end{aligned}
\end{equation*}
where $$g_1(s):=\int_{\mathbb R^3}|v_1(y,s)|^5\,{\rm d}y$$ and $\epsilon$ can be sufficiently small.
For $I_2$, we first use \eqref{miu} and \crm{Remark \ref{remark}} to conclude that there exists $\delta\in (0,T_1)$ such that the following property holds:
\begin{equation*}
\|v_2\|_{L^\infty (0,\delta;L^3(\mathbb R^3))}\leq\|v-v_0\|_{L^\infty (0,\delta;L^3(\mathbb R^3))}+\|v_1-v_0\|_{L^\infty (0,\delta;L^3(\mathbb R^3))}\leq 2\crm{\mu_1}.
\end{equation*}
Thus for $t\leq\delta$, we have
\begin{equation*}
	\begin{aligned}
		I_2&\leq\int_{0}^{t}\|v_2\|^2_{L^3(\mathbb R^3)}\|\omega\|^2_{L^6(\mathbb R^3)}\,{\rm d}s\\
		&\leq \crm{C} \|v_2\|^2_{L^\infty (0,t;L^3(\mathbb R^3))}\int_{0}^{t}\int_{\mathbb R^3}|\nabla \omega|^2\,{\rm d}x{\rm d}s\\
&\leq  4\crm{\mu_1^2 C}\int_{0}^{t}\int_{\mathbb R^3}|\nabla \omega|^2\,{\rm d}x{\rm d}s.
	\end{aligned}
\end{equation*}
Combining the estimate for $I_1$ and $I_2$, we obtain
\begin{equation*}
	\begin{aligned}
		I_1+I_2\leq c_1\int_{0}^{t}g_1(s)\int_{\mathbb R^3}|\omega(x,t)|^2\,{\rm d}x{\rm d}s+c_2\int_{0}^{t}\int_{\mathbb R^3}|\nabla \omega|^2\,{\rm d}x{\rm d}s.
	\end{aligned}
\end{equation*}
For sufficiently small $\epsilon$ and $t$, we can obtain a sufficiently small $c_2$ so that it can be absorbed by the left-hand side of \eqref{inequality}. Besides, the remaining three terms on the right-hand side of \eqref{inequality} can be estimated in the same way, \crm{and we find}
\begin{equation*}
	\begin{aligned}
		&\int_{\mathbb R^3}\big(|\omega(x,t)|^2+|\psi(x,t)|^2\big)\,{\rm d}x+\int_{0}^{t}\int_{\mathbb R^3}\big(|\nabla \omega|^2+|\nabla \psi|^2\big)\,{\rm d}x{\rm d}t\\
		&\leq C\int_{0}^{t}\big(g_1(t)+g_2(t)\big)\int_{\mathbb R^3}\big(|\omega(x,t)|^2+|\psi(x,t)|^2\big)\,{\rm d}x{\rm d}t,
	\end{aligned}
\end{equation*}
for $0<t<\delta,$ where $$g_2(t):=\int_{\mathbb R^3}|H_1(y,t)|^5\,{\rm d}y.$$
Hence, using the Gronwall's inequality, we further obtain $\omega=0,\psi=0$ on the interval $(0,\delta)$. On the other hand, by interpolation inequalities, we know that $(v,H)\in L^5(\delta,T;L^5(\mathbb R^3))$ and thus $(v_2,H_2)\in L^5(\delta,T;L^5(\mathbb R^3))$, then using the same reasoning as above, we obtain $\omega=0,\psi=0$ throughout the whole interval $(0,T)$. The theorem is thus proved.
\end{proof}
\medskip

\section*{Acknowledgments} The research of BL was partially supported by NSFC-$12371202$.

\end{document}